\documentclass[a4paper,12pt]{amsart}
\usepackage[utf8]{inputenc}
\usepackage[a4paper, top=3cm, left=2.5cm, right=2.5cm, bottom=3cm]{geometry}

\usepackage{relsize}
\usepackage{lineno}
\usepackage{hyperref}
\usepackage{amsmath,amsthm,pb-diagram,amssymb,comment}
\usepackage{amsfonts,graphicx,color}
\usepackage{enumitem, fancyhdr, dsfont}
\usepackage{thmtools,cleveref}
\usepackage[normalem]{ulem}
\usepackage{stmaryrd}
\usepackage{textcomp}
\usepackage{contour}
\usepackage{mathbbol}
\usepackage{varwidth}
\usepackage{tasks}
\usepackage{tikz,float}
\usepackage{multicol}
\usepackage{thmtools,xcolor} 

\usepackage{pifont}

\usepackage{halloweenmath}

\usepackage{fontawesome}

\hypersetup{
colorlinks=true,
linkcolor=dodger,
filecolor=dodger,
urlcolor=dodger,
citecolor=dodger,
}

\setlist{itemsep=1ex, topsep=0ex}

\allowdisplaybreaks
\newcommand{\thzfc}{\mathsf{ZFC}}
\newcommand{\thch}{\mathsf{CH}}

\newcommand{\Ed}{\mathsf{Ed}}
\newcommand{\Mbf}{\mathsf{M}}

\newcommand{\Cn}{\mathsf{Cn}}

\newcommand{\Bwf}{\mathcal{B}}

\newcommand{\Gwf}{\mathcal{G}}
\newcommand{\Hwf}{\mathcal{H}}
\newcommand{\Iwf}{\mathcal{I}}
\newcommand{\Jwf}{\mathcal{J}}
\newcommand{\Mwf}{\mathcal{M}}
\newcommand{\Nwf}{\mathcal{N}}

\newcommand{\Pwf}{\mathcal{P}}
\newcommand{\Swf}{\mathcal{S}}

\newcommand{\bfrak}{\mathfrak{b}}
\newcommand{\cfrak}{\mathfrak{c}}
\newcommand{\dfrak}{\mathfrak{d}}

\newcommand{\xfrak}{\mathfrak{x}}
\newcommand{\yfrak}{\mathfrak{y}}

\newcommand{\menos}{\smallsetminus}

\newcommand{\pts}{\mathcal{P}}
\newcommand{\frestr}{\mathord{\upharpoonright}}

\newcommand{\add}{\mathsf{add}}
\newcommand{\cov}{\mathsf{cov}}
\newcommand{\non}{\mathsf{non}}
\newcommand{\cof}{\mathsf{cof}}
\DeclareMathOperator{\limdir}{limdir}

\newcommand{\Bor}{\mathds{B}}
\newcommand{\PBor}{\mathds{PB}}
\newcommand{\Cor}{\mathds{C}}
\newcommand{\Dor}{\mathds{D}}
\newcommand{\Eor}{\mathds{E}}

\newcommand{\Loc}{\mathds{LOC}}

\newcommand{\Por}{\mathds{P}}
\newcommand{\Pbb}{\mathds{P}}
\newcommand{\Qor}{\mathds{Q}}
\newcommand{\Ior}{\mathds{I}}

\newcommand{\Ror}{\mathds{R}}

\newcommand{\Qnm}{\dot{\mathds{Q}}}

\newcommand{\SNwf}{\mathcal{SN}}

\newcommand{\R}{\mathbb{R}}

\newcommand{\cf}{\mbox{\rm cf}}

\newcommand{\sii}{{\ \mbox{$\Leftrightarrow$} \ }}

\newcommand{\la}{\langle}
\newcommand{\ra}{\rangle}

\newcommand{\sqsubm}{\sqsubset^{\rm m}}

\newcommand{\Seq}{\mathrm{seq}}
\newcommand{\Fr}{\mathsf{Fr}}
\newcommand{\Rbf}{\mathsf{R}}

\newcommand{\D}{\mathsf{D}}
\newcommand{\Cbf}{\mathsf{Cv}}
\newcommand{\Lc}{\mathsf{Lc}}

\newcommand{\aLc}{\mathsf{aLc}}
\newcommand{\Lb}{\mathsf{Lb}}
\newcommand{\Hcal}{\mathcal{H}}
\newcommand{\Scal}{\mathcal{S}}
\newcommand{\id}{\mathrm{id}}

\newcommand{\balc}{\mathfrak{b}^{\mathrm{aLc}}}
\newcommand{\dalc}{\mathfrak{d}^{\mathrm{aLc}}}

\newcommand{\scf}{\mathrm{scf}}
\newcommand{\Fn}{\mathrm{Fn}}
\newcommand{\baireincr}{\omega^{\uparrow\omega}}
\newcommand{\leqT}{\preceq_{\mathrm{T}}}
\newcommand{\eqT}{\cong_{\mathrm{T}}}

\newcommand{\gen}{\mathrm{gen}}

\newcommand{\hgt}{\mathrm{ht}}

\newcommand{\st}{\mid}
\newcommand{\set}[2]{\{#1 \st\allowbreak\,#2\}}
\newcommand{\largeset}[2]{\left\{#1 \;\middle|\; #2\right\}}
\newcommand{\seq}[2]{\la #1 \st\allowbreak\, #2\ra}

\newcommand{\baire}{\omega^\omega}
\newcommand{\cantor}{2^\omega}
\newcommand\subsetdot{\mathrel{\ooalign{$\subset$\cr
\hidewidth\hbox{$\cdot\mkern3mu$}\cr}}}
\newcommand{\supcov}{\mathsf{supcov}}
\newcommand{\supcof}{\mathsf{supcof}}
\newcommand{\minnon}{\mathsf{minnon}}
\newcommand{\minadd}{\mathsf{minadd}}
\newcommand{\Int}{\mathsf{int}}
\newcommand{\setand}{\bsp\text{and}\bsp}
\newcommand{\bsp}{\allowbreak\ }

\contourlength{0.8pt}

\definecolor{ultramarineblue}{rgb}{0.25, 0.4, 0.96}
\definecolor{cornellred}{rgb}{0.7, 0.11, 0.11}
\definecolor{cobalt}{rgb}{0.0, 0.28, 0.67}
\definecolor{bleudefrance}{rgb}{0.19, 0.55, 0.91}
\definecolor{darkblue}{rgb}{0.0, 0.0, 0.55}
\definecolor{ferrarired}{rgb}{1.0, 0.11, 0.0}
\definecolor{brandeisblue}{rgb}{0.0, 0.44, 1.0}
\definecolor{azure(colorwheel)}{rgb}{0.0, 0.5, 1.0}
\definecolor{aqua}{rgb}{0.0, 1.0, 1.0}
\definecolor{aguamarina}{cmyk}{0.85,0,0.33,0}
\definecolor{cafe}{cmyk}{0,0.81,1,0.60}
\definecolor{canela}{cmyk}{0.14,0.42,0.56,0}
\definecolor{darkgray}{cmyk}{0,0,0,0.50}
\definecolor{emerald}{cmyk}{0.91,0,0.88,0.12}
\definecolor{fresa}{cmyk}{0,1,0.50,0}
\definecolor{gold}{cmyk}{0,0.10,0.84,0}
\definecolor{lightgray}{cmyk}{0,0,0,0.30}
\definecolor{marron}{cmyk}{0,0.72,1,0.45}
\definecolor{melon}{cmyk}{0,0.29,0.84,0}
\definecolor{ladri}{cmyk}{0,0.77,0.87,0}
\definecolor{olive}{cmyk}{0.64,0,0.95,0.40}
\definecolor{orange}{cmyk}{0,0.42,1,0}
\definecolor{peach}{cmyk}{0,0.46,0.50,0}
\definecolor{pink}{cmyk}{0,0.10,0.10,0}
\definecolor{orange}{cmyk}{0,0.42,1,0}
\definecolor{pine}{cmyk}{0.92,0,0.59,0.25}
\definecolor{purple}{cmyk}{0.45,0.86,0,0}
\definecolor{violet}{cmyk}{0.07,0.90,0,0.34}
\definecolor{craneorange}{RGB}{252,187,6}
\definecolor{red(ncs)}{rgb}{0.77, 0.01, 0.2}

\definecolor{ogreen}{RGB}{107,142,35}

\usepackage{tikz}

\definecolor{sub0}{RGB}{29,32,137}
\definecolor{sub1}{RGB}{1,71,157}
\definecolor{sub2}{RGB}{1,104,183}
\definecolor{sub3}{RGB}{0,160,234}
\definecolor{sug}{RGB}{0,154,68}
\definecolor{suy}{RGB}{208,219,1}

\newcommand{\subiii}[1]{{\color{sub3}#1}}

\DeclareSymbolFont{extraup}{U}{zavm}{m}{n}
\DeclareMathSymbol{\varheart}{\mathalpha}{extraup}{86}
\DeclareMathSymbol{\vardiamond}{\mathalpha}{extraup}{87}

\definecolor{dodger}{rgb}{0.0,0.5,1.0}
\newcommand{\amber}[1]{\ifmmode\mathrel\fi{\color{amber}#1}}

\definecolor{amber}{rgb}{1.0,0.49,0.0}

\definecolor{ogreen}{RGB}{107,142,35}

\definecolor{sug}{RGB}{0,154,68}
\definecolor{sub2}{RGB}{1,104,183}

\title[]{More separations of cardinal characteristics of the strong measure zero ideal}

\author{Miguel A. Cardona}
\address[Miguel A. Cardona]{Einstein Institute of Mathematics\\
Edmond J. Safra Campus, Givat Ram\\
The Hebrew University of Jerusalem\\
Jerusalem, 91904, Israel}
\email{miguel.cardona@mail.huji.ac.il}
\urladdr{https://sites.google.com/mail.huji.ac.il/miguel-cardona-montoya/home-page}

 \author{Miroslav Repick\'y}
 \address[Miroslav Repick\'y]{Mathematical Institute\\
 Slovak Academy of Sciences\\
 Gre\v{s}\'akova 6\\
 040\,01 Ko\v{s}ice\\
 Slovak Republic}
 \email{repicky@saske.sk}

\author[S. Shelah]{Saharon Shelah}
\address[Saharon Shelah]{Einstein Institute of Mathematics,
Edmond J. Safra Campus, Givat Ram\\
The Hebrew University of Jerusalem, \\
Jerusalem, 91904, Israel; and
Department of Mathematics\\
Rutgers University\\
Piscataway, NJ 08854-8019, USA}

\email{shelah@math.huji.ac.il}
\urladdr{https://shelah.logic.at/}
\thanks{The first and third author would like to thank the Israel Science Foundation for partially supporting this research by grant 2320/23 (2023-2027); and the second author was supported by the grant VEGA 2/0104/24 of the Slovak Grant Agency VEGA and by the Slovak Research and Development Agency under the Contract no.\ APVV-20-0045}



\date{}

\definecolor{burntumber}{rgb}{0.54, 0.2, 0.14}
\definecolor{burgundy}{rgb}{0.5, 0.0, 0.13}

\numberwithin{equation}{section}

\makeatletter
\def\@roman#1{\romannumeral #1}
\makeatother

\newcounter{enuAlph}

\theoremstyle{plain}
\newtheorem{theorem}[equation]{Theorem}
\newtheorem{corollary}[equation]{Corollary}
\newtheorem{lemma}[equation]{Lemma}

\newtheorem{clm}[equation]{Claim}
\newtheorem{fact}[equation]{Fact}

\newtheorem{question}[equation]{Question}

\newtheorem{teorema}[enuAlph]{Theorem}

\theoremstyle{definition}
\newtheorem{definition}[equation]{Definition}
\newtheorem{example}[equation]{Example}
\newtheorem{remark}[equation]{Remark}

\begin{document}

\begin{abstract}
Let $\mathcal{N}$ be the $\sigma$-ideal of the null sets of reals. We introduce a new property of forcing notions that enable control of the additivity of $\mathcal{N}$ after finite support iterations.
This is applied to answer some open questions from the work of Brendle, the first author, and Mej\'ia~\cite{BCM2}.
\end{abstract}

\maketitle

\section{Introduction}\label{sec:intro}

This paper is part of the consistency results concerning the cardinal characteristics associated with the strong measure zero ideal, specifically looking into its additivity number with respect  to the bounding number and additivity number of the null ideal. We follow up on the study conducted in~\cite{BCM2} and address two remaining questions posed in that research.

The $\sigma$-ideal $\SNwf$ of the strong measure zero
subsets of the real line $\Ror$ (or of the Cantor space $\cantor$) has been receiving a lot of attention due to the discovery that \emph{Borel's conjecture}, which claims every strong measure zero set is countable, cannot be proved or disproved in $\thzfc$: $\thch$ implies that it is false and; on the other hand, R. Laver~\cite{LaverBorel} proved its consistency with $\thzfc$ using forcing. Afterward, the cardinal characteristics associated with $\SNwf$ have been a fascinating focus of study, especially in the context of Cicho\'n's diagram and other cardinal characteristics.

 Generally, these cardinals are characterized in the following manner: Let $\Iwf$ be an ideal of subsets of $X$ such that $\{x\}\in \Iwf$ for all $x\in X$. Throughout this paper, we demand that all ideals satisfy this latter requirement. We introduce the following four \emph{cardinal characteristics associated with $\Iwf$}:
\begin{align*}
 \add(\Iwf)&=\min\largeset{|\Jwf|}{\Jwf\subseteq\Iwf{,}\ \bigcup\Jwf\notin\Iwf},\\
 \cov(\Iwf)&=\min\largeset{|\Jwf|}{\Jwf\subseteq\Iwf{,}\ \bigcup\Jwf=X},\\
 \non(\Iwf)&=\min\set{|A|}{A\subseteq X{,}\ A\notin\Iwf},\textrm{\ and}\\
 \cof(\Iwf)&=\min\set{|\Jwf|}{\Jwf\subseteq\Iwf{,}\ \forall\, A\in\Iwf\ \exists\, B\in \Jwf\colon A\subseteq B}.
\end{align*}
These cardinals are referred to as the \emph{additivity, covering, uniformity} and \emph{cofinality of~$\Iwf$}, respectively. The relationship between the cardinals defined above is illustrated in \autoref{diag:idealI}.

\begin{figure}[ht!]
\centering
\begin{tikzpicture}[scale=1.2]
\small{
\node (azero) at (-1,1) {$\aleph_0$};
\node (addI) at (1,1) {$\add(\Iwf)$};
\node (covI) at (2,2) {$\cov(\Iwf)$};
\node (nonI) at (2,0) {$\non(\Iwf)$};
\node (cofI) at (4,2) {$\cof(\Iwf)$};
\node (sizX) at (4,0) {$|X|$};
\node (sizI) at (5,1) {$|\Iwf|$};

\draw (azero) edge[->] (addI);
\draw (addI) edge[->] (covI);
\draw (addI) edge[->] (nonI);
\draw (covI) edge[->] (sizX);
\draw (nonI) edge[->] (sizX);
\draw (covI) edge[->] (cofI);
\draw (nonI) edge[->] (cofI);
\draw (sizX) edge[->] (sizI);
\draw (cofI) edge[->] (sizI);
}
\end{tikzpicture}
\caption{Diagram of the cardinal characteristics associated with $\Iwf$. An arrow  $\mathfrak x\rightarrow\mathfrak y$ means that (provably in ZFC)
    $\mathfrak x\le\mathfrak y$.}
\label{diag:idealI}
\end{figure}
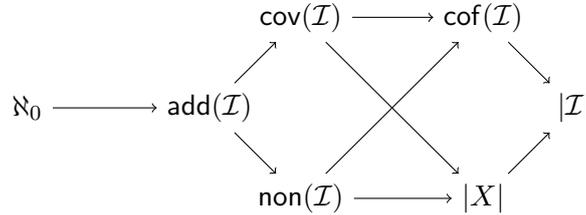

Given a formula $\phi$, $\forall^\infty\, n<\omega\colon \phi$ means that all but finitely many natural numbers satisfy~$\phi$; $\exists^\infty\, n<\omega\colon \phi$ means that infinitely many natural numbers satisfy $\phi$. For $f,g\in\baire$ define
\[f\leq^*g\text{ iff } \forall^\infty n\in\omega\colon f(n)\leq g(n).\]
We recall
\begin{align*}
    \bfrak &:=\min\set{|F|}{F\subseteq\baire\text{ and }\forall g\in\baire\ \exists f\in F:f\not\leq^* g},\\
    \dfrak &:=\min\set{|D|}{D\subseteq\baire\text{ and }\forall g\in\baire\ \exists f\in D:g\leq^* f},
\end{align*}
denote the \textit{bounding number} and the \textit{dominating number}, respectively. And denoted $\cfrak=2^{\aleph_0}$.

Denote by $\Nwf$ and $\Mwf$ the $\sigma$-ideals of Lebesgue null sets and the meager sets in~$\cantor$, respectively.  We fix some notation before defining the strong measure zero ideal.

\begin{itemize}
    \item For $s\in 2^{<\omega}$ denote $[s]:=\set{x\in\cantor}{ s\subseteq x}$.

    \item For $\sigma \in (2^{<\omega})^\omega$ define $\hgt_\sigma\colon \omega\to\omega$ by $\hgt_\sigma(n):=|\sigma(n)|$ for all $n<\omega$, which we call the \emph{height of $\sigma$}. Also, define
        \[[\sigma]_\infty:=\set{x\in\cantor}{\exists^\infty n\colon \sigma(n)\subseteq x}.\]
\end{itemize}

\begin{definition}\label{def:SN}
A set \emph{$X\subseteq\cantor$ has strong measure zero} if
\[\forall f\in\baire\, \exists \sigma\in(2^{<\omega})^\omega\colon f\leq^* \hgt_\sigma \text{ and }X\subseteq\bigcup_{i<\omega}[\sigma(i)].\]
Denote by $\SNwf$ the collection of strong measure zero subsets of $\cantor$.
\end{definition}

The cardinal characteristics associated with the strong measure zero ideal are the same for the spaces $\cantor$, $\R$, and $[0,1]$, see details in e.g.~\cite{CMR}. From now on, we work with $\SNwf=\SNwf(\cantor)$.

The following characterization of $\SNwf$ is quite practical.

\begin{lemma}\label{charSN}
Let $X\subseteq2^\omega$ and let $D\subseteq\baire$ be a dominating family. Then $X\subseteq2^\omega$ has strong measure zero in\/ $\cantor$ iff
\[\forall f\in D\, \exists\sigma\in (2^{<\omega})^\omega\colon f=\hgt_\sigma\textrm{\ and\ } X\subseteq[\sigma]_\infty.\]
\end{lemma}

Another important ideal related to $\SNwf$ is Yorioka's ideal defined below, which has played a~role significantly in studying the combinatoric of $\SNwf$.

\begin{definition}[Yorioka~{\cite{Yorioka}}]\label{DefYorio}
Set $\baireincr:=\set{d\in\omega^{\omega}}{d\text{\ is increasing}}$.
For $f\in\baireincr$ define the \emph{Yorioka ideal}
\[\Iwf_f:=\set{A\subseteq\cantor}{\exists\, \sigma\in(2^{<\omega})^\omega\colon f\ll \hgt_\sigma\text{\ and }A\subseteq[\sigma]_\infty},\]
where the relation $x \ll y$  denotes $\forall\, k<\omega\ \exists\, m_k<\omega\ \forall\, i\geq m_k\colon x(i^k)\leq y(i)$.
\end{definition}

The reason why it is used $\ll$ instead of $\leq^*$ is that the latter would not yield an ideal, as proved by Kamo and Osuga~\cite{KO08}.
Yorioka~\cite{Yorioka} has proved (indeed) that $\Iwf_f$ is a~$\sigma$-ideal when $f$ is increasing. By~\autoref{charSN} it is clear that $\SNwf=\bigcap\set{\Iwf_f}{f\in\baireincr}$ and $\Iwf_f \subseteq \Nwf$ for any $f\in\baireincr$.

Denote
\begin{align*}
    \minadd & :=\min\set{\add(\Iwf_f)}{f\in\baireincr}, &
    \supcov & :=\sup\set{\cov(\Iwf_f)}{ f\in\baireincr},\\
    \minnon & :=\min\set{\non(\Iwf_f)}{ f\in\baireincr}, &
    \supcof & :=\sup\set{\cof(\Iwf_f)}{ f\in\baireincr}.
\end{align*}
It is known from Miller that $\minnon = \non(\SNwf)$.
On the other hand, $\add(\Iwf_\id)$ is the largest of the additivities of Yorioka ideals and $\cof(\Iwf_\id)$ is the smallest of the cofinalities, similarly for $\non(\Iwf_\id)$ and $\cof(\Iwf_\id)$, where $\id$ denotes the identity function on $\omega$. See~\cite[Sec.~3]{CM} for a summary of the cardinal characteristics associated with the Yorioka ideals.

\autoref{Cichonwith_SN}~illustrates the known relations between the cardinal characteristics associated with $\SNwf$, $\Iwf$, and those in Cichoń’s diagram. See details in~\cite{Mi1982,JS89,P90,GJS,Yorioka,O08,CMR,cardona,CMR2,BCM2}. The quoted references also display that the diagram is mostly complete, however, there are still unanswered questions,  for instance, it is not known whether $\bfrak$, $\dfrak$, $\non(\Nwf)$, $\supcof$ and $\cof(\Nwf)$ are lower bounds of $\cof(\SNwf)$ (see~\cite[Sec.~8]{BCM2}), and whether $\add(\Nwf)=\minadd$ and $\cof(\Nwf)=\supcof$ (see~\cite[Sec.~6]{CM}).

\begin{figure}[ht]
\begin{center}
\includegraphics[width=\linewidth]{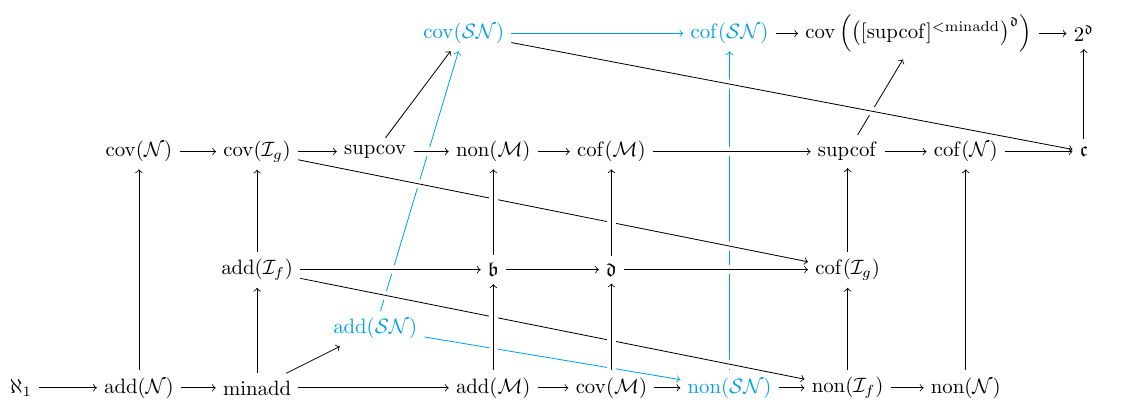}
\caption{
An arrow $\xfrak\rightarrow\yfrak$ means that
(provably in ZFC) $\xfrak\le\yfrak$.
Moreover, $\add(\Mwf)=\min\{\bfrak,\non(\SNwf)\}$ and $\cof(\Mwf)=\max\{\dfrak,\supcov\}$. The ideals $\Iwf_f$ and~$\Iwf_g$ are illustrated for arbitrary $f$ and $g$ to emphasize that the covering of any Yorioka ideal is an upper bound of the additivities of all Yorioka ideals, likewise for the cofinality and uniformity.}
\label{Cichonwith_SN}
\end{center}
\end{figure}

Regarding the consistency results, it was raised in~\cite[Sec.~8]{BCM2}:
\begin{enumerate}[beginpenalty=1000,midpenalty=10000,label=\rm(Q\arabic*)]
\item\label{addSN1} $\add(\Nwf) < \add(\SNwf) < \bfrak$.
\item\label{addSN2} $\add(\Nwf) < \bfrak < \add(\SNwf)$.
\end{enumerate}

Roughly speaking, in the context of FS iterations, only two approaches are known not to increase $\add(\Nwf)$: $\sigma$-centered posets do not increase $\add(\Nwf)$ as shown by Judah and Shelah~\cite{JS}, with improvements by Brendle~\cite{Br}, however, this is not enough to solve~\ref{addSN1}--\ref{addSN2}, because such posets can not be able to increase $\add(\SNwf)$ as discovered by Brendle, the first author, and Mej\'ia~\cite{BCM2}; and  $(\rho,\varrho)$-linked posets for a suitable pair $(\rho,\varrho)$ (a property between $\sigma$-centered and $\sigma$-linked, defined in~\cite{KO}, see~\autoref{link}) do not increase $\add(\Nwf)$ as proved by Brendle and Mej\'ia~\cite{BrM}. Consequently, to solve~\ref{addSN1}--\ref{addSN2}, the latter linked property seems to be the only way to address it. In connection with~\ref{addSN1}--\ref{addSN2}, the first author, Mej\'ia, and Rivera-Madrid introduced a~new $\sigma$-$n$-linked poset $\Qor_f$ for an increasing function $f\in\baire$ (\autoref{c5}) that increases $\add(\SNwf)$, but it is not clear whether such poset satisfies the $(\rho,\varrho)$-linked.

In this paper, we solve~\ref{addSN1}--\ref{addSN2}: we provide a new linked property that enables us to force the additivity of~$\Nwf$ small. Concretely, we introduce a new weaker linked property than $(\rho,\varrho)$-linked, called $\sigma$-$\bar\rho$-linked for a suitable sequence of functions $\bar\rho=\seq{\rho_n}{n\in\omega}$ in $\baire$ and prove that this property is good for keeping $\add(\Nwf)$ small in forcing
extensions.

\begin{teorema}[\autoref{d3}]\label{Thm:a4}
 Assume
that $h\in\baire$ converges to infinity. Then there are some $\Hwf_{\bar\rho,h}=\set{h_n}{n\in\omega}\subseteq\baire$ with $h_0=h$ such that any $\sigma$-$\bar\rho$-linked posets is $\Lc^*(\Hwf_{\bar\rho,h})$-good.
\end{teorema}

The inspiration for the preceding theorem comes from~\cite[Lem.~5.14]{BrM}. In~\autoref{sec:s4}, we introduce the $\sigma$-$\bar\rho$-linked property and give some examples: such a random forcing, $\Qor_f$, and $\sigma$-centered posets; and we prove~\autoref{Thm:a4}.

\autoref{Thm:a4} is used to prove our main consistency results regarding the separation of cardinal characteristics of $\SNwf$ and cardinals in Cicho\'n’s diagram simultaneously.

Our main results are as follows.

\begin{teorema}[\autoref{appl:I}]\label{Thm:a0}
Let $\theta_1\leq\theta_2\leq\theta_3\leq\theta_4\leq\theta_5\leq \theta_6\leq\theta_7=\theta_7^{\aleph_0}$ be uncountable regular cardinals and assume that $\cof([\theta_7]^{<\theta_i})=\theta_7$ for $1 \leq i \leq 5$.
Then we can construct a~ccc poset that forces~\autoref{Fig:addn0}.

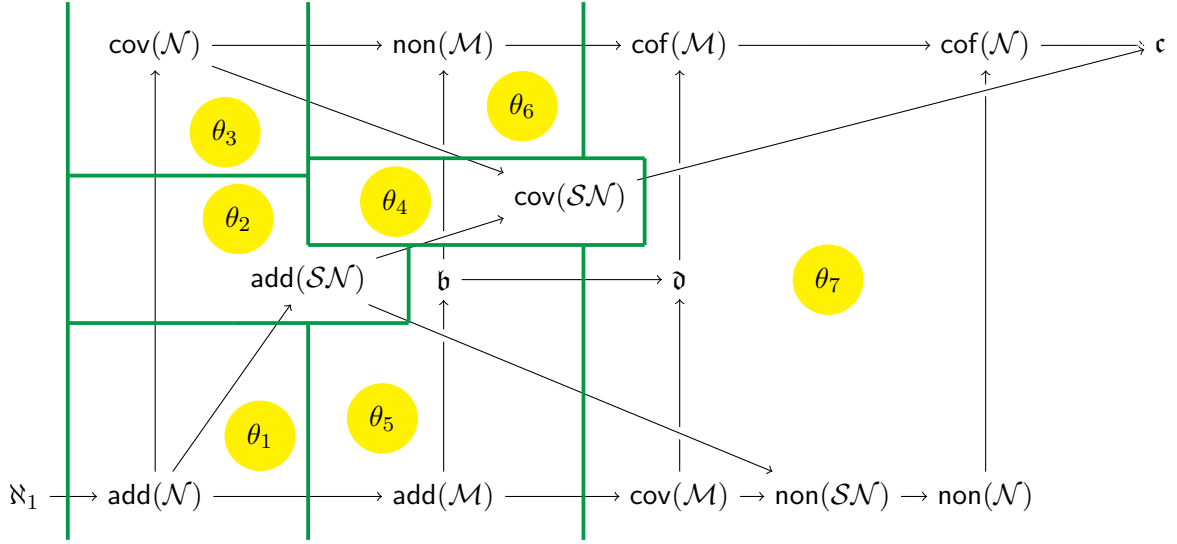
\begin{figure}[ht!]
\centering
\begin{tikzpicture}[scale=1.15]
\small{
\node (aleph1) at (-1,2.3) {$\aleph_1$};
\node (addn) at (0.5,2.3){$\add(\Nwf)$};
\node (covn) at (0.5,7.5){$\cov(\Nwf)$};
\node (nonn) at (10,2.3) {$\non(\Nwf)$} ;
\node (cfn) at (10,7.5) {$\cof(\Nwf)$} ;
\node (addm) at (3.8,2.3) {$\add(\Mwf)$} ;
\node (covm) at (6.5,2.3) {$\cov(\Mwf)$} ;
\node (nonm) at (3.8,7.5) {$\non(\Mwf)$} ;
\node (cfm) at (6.5,7.5) {$\cof(\Mwf)$} ;
\node (b) at (3.8,4.8) {$\bfrak$};
\node (d) at (6.5,4.8) {$\dfrak$};
\node (c) at (12,7.5) {$\cfrak$};
\node (covsn) at (5.25,5.75) {$\cov(\SNwf)$};
\node (addsn) at (2.25,4.8) {$\add(\SNwf)$};

\node (nonsn) at (8.25,2.3) {$\non(\SNwf)$} ;

\draw (aleph1) edge[->] (addn)
      (addn) edge[->] (covn)
      (covn) edge [->] (nonm)
      (nonm)edge [->] (cfm)
      (cfm)edge [->] (cfn)
      (cfn) edge[->] (c);

\draw   (covm) edge [->]  (nonsn)
(nonsn) edge [->]  (nonn)
(addm) edge [->]  (covm)
(addn) edge [->]  (addm)
(addn) edge [->]  (addsn)
(nonn) edge [->]  (cfn)
 (addm) edge [->] (b)
      (b)  edge [->] (nonm);
\draw (covm) edge [->] (d)
      (d)  edge[->] (cfm);
\draw (b) edge [->] (d);

\draw   (addsn) edge [line width=.15cm,white,-]  (covsn)
(addsn) edge [->]  (covsn);
\draw   (covsn) edge [line width=.15cm,white,-]  (c)
(covsn) edge [->]  (c);
\draw   (addsn) edge [line width=.15cm,white,-]  (nonsn)
(addsn) edge [->]  (nonsn);
(covsn) edge [->]  (c);

\draw   (covn) edge [line width=.15cm,white,-]  (covsn)
(covn) edge [->]  (covsn);
\draw[color=sug,line width=.05cm] (-0.5,1.8)--(-0.5,8);
\draw[color=sug,line width=.05cm] (-0.5,4.3)--(1.4,4.3);
\draw[color=sug,line width=.05cm] (2.25,1.8)--(2.25,4.3);
\draw[color=sug,line width=.05cm] (1.4,4.3)--(3.4,4.3);
\draw[color=sug,line width=.05cm] (3.4,4.3)--(3.4,5.2);
\draw[color=sug,line width=.05cm] (2.25,5.2)--(6.1,5.2);
\draw[color=sug,line width=.05cm] (2.25,5.2)--(2.25,8);
\draw[color=sug,line width=.05cm] (5.4,6.2)--(5.4,8);
\draw[color=sug,line width=.05cm] (-0.5,6)--(2.25,6);
\draw[color=sug,line width=.05cm] (2.25,6.2)--(6.1,6.2);
\draw[color=sug,line width=.05cm] (6.1,5.2)--(6.1,6.2);
\draw[color=sug,line width=.05cm] (5.4,1.8)--(5.4,5.2);
\draw[circle, fill=yellow,color=yellow] (1.7,3) circle (0.4);
\draw[circle, fill=yellow,color=yellow] (3.1,3.2) circle (0.4);
\draw[circle, fill=yellow,color=yellow] (3.25,5.7) circle (0.4);
\draw[circle, fill=yellow,color=yellow] (1.45,5.5) circle (0.4);
\draw[circle, fill=yellow,color=yellow] (1.3,6.5) circle (0.4);
\draw[circle, fill=yellow,color=yellow] (4.7,6.8) circle (0.4);
\draw[circle, fill=yellow,color=yellow] (8.2,4.8) circle (0.4);
\node at (1.7,3) {$\theta_1$};
\node at (3.1,3.2) {$\theta_5$};
\node at (3.25,5.7) {$\theta_4$};
\node at (1.3,6.5) {$\theta_3$};
\node at (1.45,5.5) {$\theta_2$};
\node at (4.7,6.8) {$\theta_6$};
\node at (8.2,4.8) {$\theta_7$};
}

\end{tikzpicture}
\caption{Constellation forced in~\autoref{Thm:a0}.}
\label{Fig:addn0}
\end{figure}
\end{teorema}

\begin{teorema}\label{Thm:a2}
 Let $\theta_1\leq\theta_2\leq\theta_3\leq\theta_4 \leq\theta_5\leq \theta_6\leq\theta_7$ be uncountable regular cardinals, and $\theta_8=\theta_8^{\aleph_0}\geq \theta_7$ satisfying\/ $\cof([\theta_8]^{<\theta_i}) = \theta_8$ for\/ $1\leq i\leq 5$. Then we can construct a~ccc poset that forces~\autoref{Fig:addn1}.
 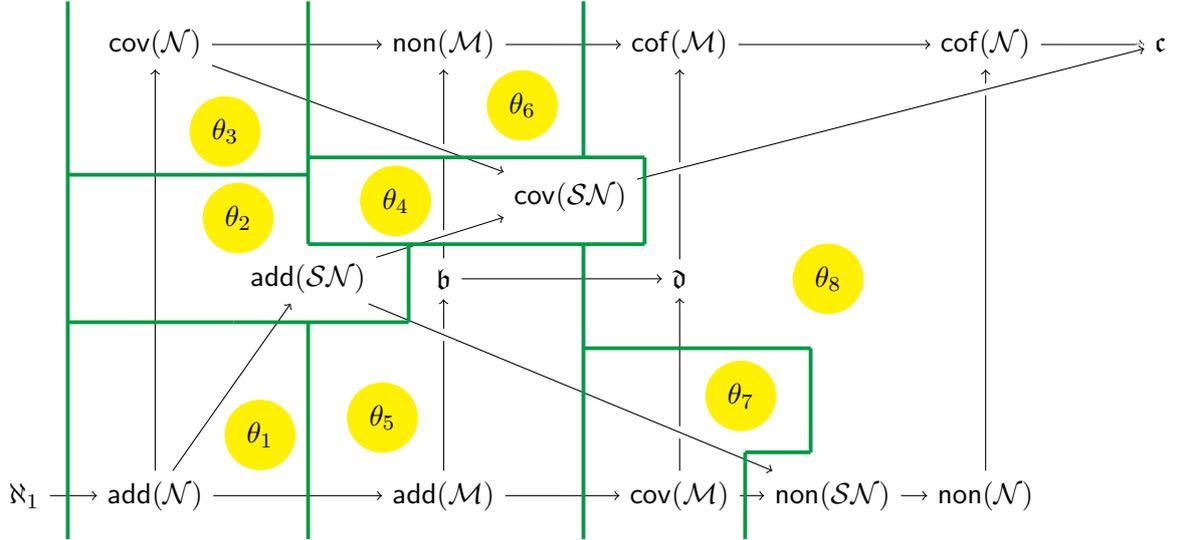
\begin{figure}[ht!]
\centering
\begin{tikzpicture}[scale=1.15]
\small{
\node (aleph1) at (-1,2.3) {$\aleph_1$};
\node (addn) at (0.5,2.3){$\add(\Nwf)$};
\node (covn) at (0.5,7.5){$\cov(\Nwf)$};
\node (nonn) at (10,2.3) {$\non(\Nwf)$} ;
\node (cfn) at (10,7.5) {$\cof(\Nwf)$} ;
\node (addm) at (3.8,2.3) {$\add(\Mwf)$} ;
\node (covm) at (6.5,2.3) {$\cov(\Mwf)$} ;
\node (nonm) at (3.8,7.5) {$\non(\Mwf)$} ;
\node (cfm) at (6.5,7.5) {$\cof(\Mwf)$} ;
\node (b) at (3.8,4.8) {$\bfrak$};
\node (d) at (6.5,4.8) {$\dfrak$};
\node (c) at (12,7.5) {$\cfrak$};
\node (covsn) at (5.25,5.75) {$\cov(\SNwf)$};
\node (addsn) at (2.25,4.8) {$\add(\SNwf)$};

\node (nonsn) at (8.25,2.3) {$\non(\SNwf)$} ;

\draw (aleph1) edge[->] (addn)
      (addn) edge[->] (covn)
      (covn) edge [->] (nonm)
      (nonm)edge [->] (cfm)
      (cfm)edge [->] (cfn)
      (cfn) edge[->] (c);

\draw   (covm) edge [->]  (nonsn)
(nonsn) edge [->]  (nonn)
(addm) edge [->]  (covm)
(addn) edge [->]  (addm)
(addn) edge [->]  (addsn)
(nonn) edge [->]  (cfn)
 (addm) edge [->] (b)
      (b)  edge [->] (nonm);
\draw (covm) edge [->] (d)
      (d)  edge[->] (cfm);
\draw (b) edge [->] (d);

\draw   (addsn) edge [line width=.15cm,white,-]  (covsn)
(addsn) edge [->]  (covsn);
\draw   (covsn) edge [line width=.15cm,white,-]  (c)
(covsn) edge [->]  (c);
\draw   (addsn) edge [line width=.15cm,white,-]  (nonsn)
(addsn) edge [->]  (nonsn);
(covsn) edge [->]  (c);

\draw   (covn) edge [line width=.15cm,white,-]  (covsn)
(covn) edge [->]  (covsn);
\draw[color=sug,line width=.05cm] (-0.5,1.8)--(-0.5,8);
\draw[color=sug,line width=.05cm] (-0.5,4.3)--(1.4,4.3);
\draw[color=sug,line width=.05cm] (2.25,1.8)--(2.25,4.3);
\draw[color=sug,line width=.05cm] (1.4,4.3)--(3.4,4.3);
\draw[color=sug,line width=.05cm] (3.4,4.3)--(3.4,5.2);
\draw[color=sug,line width=.05cm] (2.25,5.2)--(6.1,5.2);
\draw[color=sug,line width=.05cm] (2.25,5.2)--(2.25,8);
\draw[color=sug,line width=.05cm] (5.4,6.2)--(5.4,8);
\draw[color=sug,line width=.05cm] (-0.5,6)--(2.25,6);
\draw[color=sug,line width=.05cm] (2.25,6.2)--(6.1,6.2);
\draw[color=sug,line width=.05cm] (6.1,5.2)--(6.1,6.2);
\draw[color=sug,line width=.05cm] (5.4,1.8)--(5.4,5.2);
\draw[color=sug,line width=.05cm] (5.4,4)--(8,4);
\draw[color=sug,line width=.05cm] (8,2.8)--(8,4);
\draw[color=sug,line width=.05cm] (7.25,2.8)--(8,2.8);
\draw[color=sug,line width=.05cm] (7.25,1.8)--(7.25,2.8);
\draw[circle, fill=yellow,color=yellow] (1.7,3) circle (0.4);
\draw[circle, fill=yellow,color=yellow] (3.1,3.2) circle (0.4);
\draw[circle, fill=yellow,color=yellow] (3.25,5.7) circle (0.4);
\draw[circle, fill=yellow,color=yellow] (1.45,5.5) circle (0.4);
\draw[circle, fill=yellow,color=yellow] (1.3,6.5) circle (0.4);
\draw[circle, fill=yellow,color=yellow] (4.7,6.8) circle (0.4);
\draw[circle, fill=yellow,color=yellow] (7.2,3.45) circle (0.4);
\draw[circle, fill=yellow,color=yellow] (8.2,4.8) circle (0.4);
\node at (1.7,3) {$\theta_1$};
\node at (3.1,3.2) {$\theta_5$};
\node at (3.25,5.7) {$\theta_4$};
\node at (1.3,6.5) {$\theta_3$};
\node at (1.45,5.5) {$\theta_2$};
\node at (4.7,6.8) {$\theta_6$};
\node at (8.2,4.8) {$\theta_8$};
\node at (7.2,3.45) {$\theta_7$};
}

\end{tikzpicture}
\caption{Constellation forced in~\autoref{Thm:a2}.}
\label{Fig:addn1}
\end{figure}
\end{teorema}

\begin{teorema}[\autoref{appl:III}]\label{Thm:a1}
Let $\theta_1\leq\theta_2\leq\theta_3\leq\theta_4$ be uncountable regular cardinals, and $\theta_5$ a cardinal such that $\theta_5\geq\theta_4$ and\/ $\cof([\theta_5]^{<\theta_i})=\theta_5$ for $i=1,2$ and $\theta_5 = \theta_5^{\aleph_0}$. Then we can construct a~ccc poset that forces~\autoref{Fig:addn}.

\begin{figure}[ht!]
\centering
\begin{tikzpicture}[scale=1.15]
\small{
\node (aleph1) at (-1,2.3) {$\aleph_1$};
\node (addn) at (0.5,2.3){$\add(\Nwf)$};
\node (covn) at (0.5,7.5){$\cov(\Nwf)$};
\node (nonn) at (10,2.3) {$\non(\Nwf)$} ;
\node (cfn) at (10,7.5) {$\cof(\Nwf)$} ;5
\node (addm) at (3.8,2.3) {$\add(\Mwf)$} ;
\node (covm) at (6.5,2.3) {$\cov(\Mwf)$} ;
\node (nonm) at (3.8,7.5) {$\non(\Mwf)$} ;
\node (cfm) at (6.5,7.5) {$\cof(\Mwf)$} ;
\node (b) at (3.8,4.8) {$\bfrak$};
\node (d) at (6.5,4.8) {$\dfrak$};
\node (c) at (12,7.5) {$\cfrak$};
\node (covsn) at (5.25,5.75) {$\cov(\SNwf)$};
\node (addsn) at (2.25,4.8) {$\add(\SNwf)$};

\node (nonsn) at (8.25,2.3) {$\non(\SNwf)$} ;

\draw (aleph1) edge[->] (addn)
      (addn) edge[->] (covn)
      (covn) edge [->] (nonm)
      (nonm)edge [->] (cfm)
      (cfm)edge [->] (cfn)
      (cfn) edge[->] (c);

\draw   (covm) edge [->]  (nonsn)
(nonsn) edge [->]  (nonn)
(addm) edge [->]  (covm)
(addn) edge [->]  (addm)
(addn) edge [->]  (addsn)
(nonn) edge [->]  (cfn)
 (addm) edge [->] (b)
      (b)  edge [->] (nonm);
\draw (covm) edge [->] (d)
      (d)  edge[->] (cfm);
\draw (b) edge [->] (d);

\draw   (addsn) edge [line width=.15cm,white,-]  (covsn)
(addsn) edge [->]  (covsn);
\draw   (covsn) edge [line width=.15cm,white,-]  (c)
(covsn) edge [->]  (c);
\draw   (addsn) edge [line width=.15cm,white,-]  (nonsn)
(addsn) edge [->]  (nonsn);
(covsn) edge [->]  (c);

\draw   (covn) edge [line width=.15cm,white,-]  (covsn)
(covn) edge [->]  (covsn);
\draw[color=sug,line width=.05cm] (-0.5,1.8)--(-0.5,8);
\draw[color=sug,line width=.05cm] (-0.5,4.3)--(1.4,4.3);
\draw[color=sug,line width=.05cm] (2.25,1.8)--(2.25,4.3);
\draw[color=sug,line width=.05cm] (1.4,4.3)--(3.4,4.3);
\draw[color=sug,line width=.05cm] (3.4,4.3)--(3.4,5.2);
\draw[color=sug,line width=.05cm] (3.4,5.2)--(6.1,5.2);
\draw[color=sug,line width=.05cm] (5.4,6.2)--(5.4,8);
\draw[color=sug,line width=.05cm] (5.4,6.2)--(6.1,6.2);
\draw[color=sug,line width=.05cm] (6.1,5.2)--(6.1,6.2);
\draw[color=sug,line width=.05cm] (5.4,1.8)--(5.4,5.2);
\draw[color=sug,line width=.05cm] (5.4,4)--(12,4);
\draw[circle, fill=yellow,color=yellow] (1.7,3) circle (0.4);
\draw[circle, fill=yellow,color=yellow] (3.1,3.2) circle (0.4);
\draw[circle, fill=yellow,color=yellow] (2.25,5.8) circle (0.4);
\draw[circle, fill=yellow,color=yellow] (8.2,3.2) circle (0.4);
\draw[circle, fill=yellow,color=yellow] (8.2,5.8) circle (0.4);
\node at (1.7,3) {$\theta_1$};
\node at (3.1,3.2) {$\theta_2$};
\node at (2.25,5.8) {$\theta_3$};
\node at (8.2,3.2) {$\theta_4$};
\node at (8.2,5.8) {$\theta_5$};
}

\end{tikzpicture}
\caption{Constellation forced in~\autoref{Thm:a1}.}
\label{Fig:addn}
\end{figure}
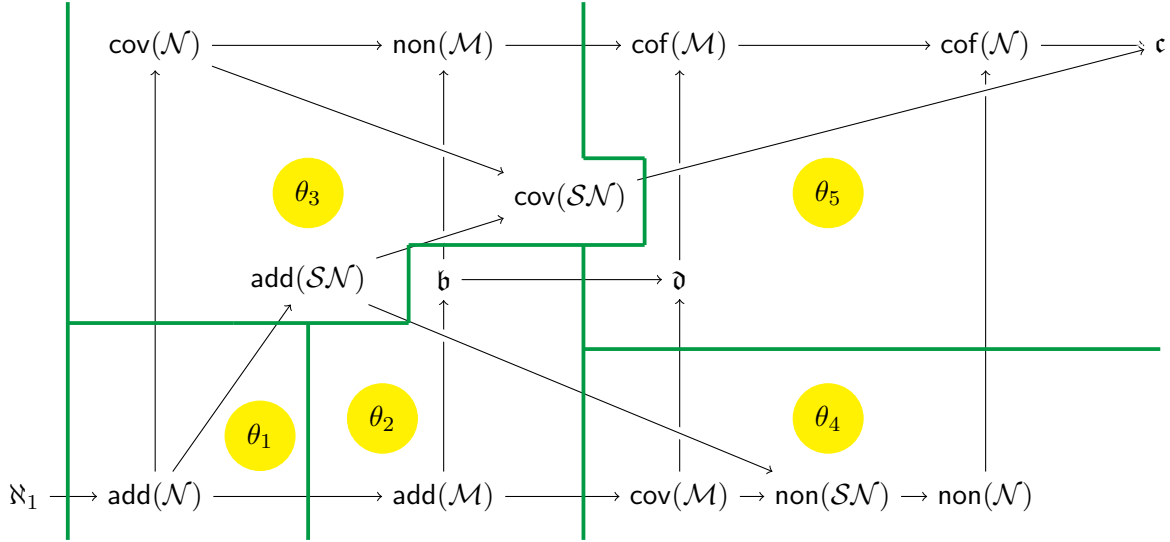
\end{teorema}

The~\emph{ultrafilter extendable matrix iteration} (uf-extendable matrix iteration, see~\autoref{f4}) is a useful powerful method to separate the left-hand side of Cicho\'n's diagram from~\cite{BCM} by incorporating
$\cov(\Mwf)<\dfrak=\non(\Nwf)=\cfrak$. This involves building ultrafilters along a matrix iteration instead of constructing sequences of ultrafilters along a~FS iteration,  as applied by Goldstern, Mej\'ia, and the third author to ensure that restrictions of the eventually different real forcing do not add dominating reals to force the consistency of the left-hand side of Cicho\'n's diagram.

Since the establishment of ultrafilter extendable matrix iteration,  it has been applied several times to separate cardinals on the left side of Cicho\'n's diagram along with other cardinal characteristics.  See~\cite{Car23,BCM2,CMR2} for details.

We will employ the method of the uf-extendable matrix iteration to prove~\autoref{Thm:a0},~\ref{Thm:a2} and~\ref{Thm:a1}. Details are provided in~\autoref{sec:s5}.

\textit{This paper is structured as follows}. In~\autoref{sec:s1}, we review all the essentials related to relational systems and the Tukey order, forcing notions and their related to properties. We review in~\autoref{sec:s3} the preservation theory of unbounded families from~\cite{CM}, which is a generalization of Judah's and third author~\cite{JS} and Brendle~\cite{Br} preservation theory. \autoref{sec:s4} is devoted to our new linkedness property and to prove~\autoref{Thm:a4}. Using~\autoref{d3}, in \autoref{sec:s5}, we prove~\autoref{Thm:a0},~\ref{Thm:a2} and~\ref{Thm:a1}. Lastly,  discussions and open questions are presented in \autoref{sec:s6}.

\section{Preliminaries}\label{sec:s1}

Most of our notation is standard and compatible with classical textbooks on Set Theory (like~\cite{kunen80,Ke2,Je2,kunen}).

We say that $\Rbf=\la X, Y, {\sqsubset}\ra$ is a~\textit{relational system} if it consists of two non-empty sets $X$ and $Y$ and a~relation~$\sqsubset$.
\begin{enumerate}[label=(\arabic*)]
\item A~set $F\subseteq X$ is \emph{$\Rbf$-bounded} if $\exists y\in Y\,\forall x\in F\,\colon x \sqsubset y$.
\item A~set $D\subseteq Y$ is \emph{$\Rbf$-dominating} if $\forall x\in X\,\exists y\in D\,\colon x \sqsubset y$.
\end{enumerate}

We associate two cardinal characteristics with this relational system $R$:
\begin{align*}
\bfrak(\Rbf)&:=\min\set{|F|}{\text{$F\subseteq X$ is $\Rbf$-unbounded}},
&&\text{the \emph{unbounding number of\/ $R$}, and}\\
\dfrak(\Rbf)&:=\min\set{|D|}{\text{$D\subseteq Y$ is $\Rbf$-dominating}},
&&\text{the \emph{dominating number of\/ $R$}.}
\end{align*}
The dual of $\Rbf$ is defined by $\Rbf^\perp:=\la Y, X, {\sqsubset^\perp}\ra$ where $y\sqsubset^\perp x$ iff $x\not\sqsubset y$. Note that $\bfrak(\Rbf^\perp)=\dfrak(\Rbf)$ and $\dfrak(R^\perp)=\bfrak(\Rbf)$.

The cardinal $\bfrak(\Rbf)$ may be undefined, in which case we write $\bfrak(\Rbf) = \infty$, as well as for~$\dfrak(\Rbf)$. Concretely, $\bfrak(\Rbf) = \infty$ iff $\dfrak(\Rbf) =1$; and $\dfrak(\Rbf)= \infty$ iff $\bfrak(\Rbf) =1$.

As in~\cite{GKMS,CM22}, we also look at relational systems given by directed preorders.
\begin{definition}\label{examSdir}
We say that $\la S,{\leq_S}\ra$ is a \emph{directed preorder} if it is a preorder (that is, $\leq_S$~is a~reflexive and transitive relation in $S$) such that
\[\forall\, x, y\in S\ \exists\, z\in S\colon x\leq_S z\text{ and }y\leq_S z.\]
A directed preorder $\la S,{\leq_S}\ra$ is seen as the relational system $S=\la S, S,{\leq_S}\ra$, and their associated cardinal characteristics are indicated by $\bfrak(S)$ and $\dfrak(S)$. The cardinal $\dfrak(S)$ is actually the \emph{cofinality of $S$}, typically denoted by $\cof(S)$ or $\cf(S)$.
\end{definition}

Note that $\leq^*$ is a directed preorder on $\baire$, where $x\leq^* y$ means $\forall^\infty n<\omega\colon x(n)\leq y(n)$. We think of $\baire$ as the relational system with the relation $\leq^*$. Then $\bfrak:=\bfrak(\baire)$ and $\dfrak:=\dfrak(\baire)$.

Relational systems can also characterize the cardinal characteristics associated with an ideal.

\begin{example}\label{exm:Iwf}
For $\Iwf\subseteq\pts(X)$ define the relational systems:
\begin{enumerate}[label=\rm(\arabic*)]
\item $\Iwf:=\la\Iwf,\Iwf,{\subseteq}\ra$, which is a~directed partial order when $\Iwf$ is closed under unions (e.g.\ an ideal).
\item $\Cbf_\Iwf:=\la X,\Iwf,{\in}\ra$.
\end{enumerate}
Whenever $\Iwf$ is an ideal on $X$,
\begin{tasks}
[label-format={\normalfont},label={(\alph*)},label-align=right,
item-indent=38pt,label-width=20pt,label-offset=5pt,
after-item-skip=6pt](2)
\task$\bfrak(\Iwf)=\add(\Iwf)$. 
\task$\dfrak(\Iwf)=\cof(\Iwf)$. 
\task$\bfrak(\Cbf_\Iwf)=\non(\Iwf)$. 
\task$\dfrak(\Cbf_\Iwf)=\cov(\Iwf)$. 
\end{tasks}
\end{example}

The Tukey connection is a useful tool for establishing relationships between cardinal characteristics. Let $\Rbf=\la X,Y,{\sqsubset}\ra$ and $\Rbf'=\la X',Y',{\sqsubset'}\ra$ be two relational systems. We say that
$(\Psi_-,\Psi_+)\colon \Rbf\to \Rbf'$
is a~\emph{Tukey connection from $\Rbf$ into $\Rbf'$} if
 $\Psi_-\colon X\to X'$ and $\Psi_+\colon Y'\to Y$ are functions such that
 \[
 \forall x\in X\,\forall y'\in Y'\colon
 \Psi_-(x) \sqsubset' y' \Rightarrow x \sqsubset \Psi_+(y').
 \]
The \emph{Tukey order} between relational systems is defined by
$\Rbf\leqT \Rbf'$ iff there is a~Tukey connection from $\Rbf$ into $\Rbf'$. \emph{Tukey equivalence} is defined by $\Rbf\eqT \Rbf'$ iff $\Rbf\leqT \Rbf'$ and $\Rbf'\leqT \Rbf$. Notice that $\Rbf\leqT \Rbf'$ implies $(\Rbf')^\perp\leqT \Rbf^\perp$, $\dfrak(\Rbf)\leq\dfrak(\Rbf')$ and $\bfrak(\Rbf')\leq\bfrak(\Rbf)$.

In relation to the relational systems $\Cbf_{\Mwf}$ and $\Cbf_\SNwf$ and their cardinal characteristics, we introduce the anti-localization relational systems and their cardinals.

\begin{definition}
 Let $b$ be a function with domain $\omega$ such that $b(i)\neq\emptyset$ for all $i<\omega$, and let $h\in\baire$. Denote $\aLc(b,h):=\la\Scal(b,h),\prod b,{\not\ni^\infty}\ra$ the relational system where $x\in^\infty\varphi$ iff $\exists^\infty\, n<\omega\colon x(n)\in \varphi(n)$ ($\aLc$ stands for \emph{anti-localization}). The \emph{anti-localization cardinals} are defined by $\balc_{b,h}:=\bfrak(\aLc(b,h))$ and $\dalc_{b,h}:=\dfrak(\aLc(b,h))$.
\end{definition}

Recall the following characterization of the cardinal characteristics associated with $\Mwf$ whenever $h\geq^* 1$, $\balc_{\omega,h}=\non(\Mwf)$ and $\dalc_{\omega,h}=\cov(\Mwf)$ (see \autoref{b0}~\ref{b0:2}), where $\omega$~denotes the constant sequence $\omega$.


\begin{lemma}[{\cite[Lem.~2.5]{KM21}}]\label{c8}
$\Cbf_{\SNwf}^\perp \leqT \Cbf^\perp_{\Iwf_g} \leqT \aLc(b,1)$ for some $g\in\baireincr$. In particular, $\balc_{b,1}\leq\cov(\SNwf)$ and\/ $\non(\SNwf)\leq\balc_{b,1}$.
\end{lemma}

In the following, we review the elements and notation of forcing theory that we use throughout this text. We denote the ground model by $V$. For two forcing notions $\Por$ and~$\Qor$, we write $\Por\subsetdot\Qor$ when $\Por$ is a complete suborder of $\Qor$, i.e.\ the inclusion map from $\Por$ into $\Qor$ is a complete embedding. When $\seq{\Por_\alpha}{\alpha\leq\beta}$ is a $\subsetdot$-increasing sequence of forcing notions (like an iteration) and $G$ is a $\Por_\beta$-generic over $V$, we denote, for $\alpha\leq\beta$, $G_\alpha= \Por_\alpha\cap G$ and $V_\alpha  =V[G_\alpha]$. When $\Por_{\alpha+1}$ is obtained by a two-step iteration $\Por_\alpha\ast\Qnm_\alpha$, $G(\alpha)$ denotes the $\Qnm[G_\alpha]$ -generic set on $V_\alpha$ such that $V_{\alpha+1}= V_\alpha[G(\alpha)]$ (that is, $G_{\alpha+1} = G_\alpha\ast G(\alpha)$). We use $\Vdash_\alpha$ to denote the forcing relation for $\Por_\alpha$, and $\leq_\alpha$ to denote its preorder (although we use $\leq$ when clear from context).

Next, we recall the following stronger versions of the chain condition of a forcing notion.

\begin{definition}
Let $\Por$ be a forcing notion and $\nu$ an infinite cardinal.
\begin{enumerate}[label=\rm(\arabic*)]
\item \emph{$\Pbb$ has the $\nu$-cc (the $\nu$-chain condition)} if every antichain in $\Por$ has size ${<}\nu$. \textit{$\Por$ has the ccc (the countable chain contidion)} if it has the $\aleph_{1}$-cc.
\item For $n<\omega$ a~set $B\subseteq \Por$ is $n$-\textit{linked} if, for every $F\subseteq B$ of size $\leq n$, $\exists q\in \Por$ $\forall p\in F\colon q\leq p$. When $n=2$ we just write \emph{linked}.
\item $C\subseteq \Por$ is  \textit{centered} if it is $n$-linked for every $n<\omega$.
\item $\Por$ is $\nu$-$n$-\textit{linked} if $\Por=\bigcup_{\alpha<\nu}P_\alpha$ where each $P_{\alpha}$ is $n$-linked. When $\nu=\omega$, we say that $\Pbb$ is $\sigma$-$n$-\textit{linked}. In the case $n=2$, we just write \emph{$\nu$-linked} and \emph{$\sigma$-linked}.
\item $\Por$ is $\nu$-\textit{centered} if $\Por=\bigcup_{\alpha<\nu}P_\alpha$ where each $P_{\alpha}$ is centered. When $\nu=\omega$, we say that $\Por$ is $\sigma$-\textit{centered}.
\end{enumerate}
\end{definition}

The following linkedness property is a generalization of the notion of $\sigma$-linkedness, which is owing to Kamo and Osuga.

\begin{definition}[{\cite{KO}}]\label{link}
Let $\rho,\varrho\in\baire$. A forcing notion $\Pbb$ is \textit{$(\rho,\varrho)$-linked} if there exists a sequence $\seq{Q_{n,j}}{n<\omega,\ j<\rho(n)}$ of subsets of $\Pbb$ such that
\begin{enumerate}[label= \rm (\roman*)]
\item\label{it:rhopi1} $Q_{n,j}$ is $\varrho(n)$-linked for all $n<\omega$ and $j<\rho(n)$, and
\item\label{it:rhopi2} $\forall\, p\in \Pbb\ \forall^{\infty}\, n<\omega\ \exists\, j<\rho(n)\colon p\in Q_{n,j}$.
\end{enumerate}
\end{definition}

Recall the following forcing notions: In order to increase $\add(\SNwf)$ in our main results, we use the forcing notion introduced in~\cite{CMR2}, which is weakening of a forcing of Yorioka~\cite[Def.~4.1]{Yorioka}.

\begin{definition}[{\cite[Def.~3.23]{CMR2}}]\label{c5}
Let $f$ be an increasing function in $\baire$.
Define~$\Qor_{f}$ as the poset whose conditions are triples $(\sigma, N, F)$ such that $\sigma\in (2^{<\omega})^{<\omega}$, $N<\omega$ and $F\subseteq (2^{<\omega})^\omega$, satisfying the following requirements:
\begin{itemize}
\item $|\sigma(i)|=f(i)$ for all $i<|\sigma|$,

\item $|F|\leq N$ and $|\sigma|\leq N^2$, and 

\item $\forall\tau\in F\,\forall n<\omega\colon |\tau(n)|=f((n+1)^2)$.
\end{itemize}
We order $\Qor_{f}$ by $(\sigma', N', F')\leq(\sigma, N, F)$ iff $\sigma\subseteq\sigma'$, $N\leq N'$, $F\subseteq F'$ and \[\forall\tau\in F\, \forall i\in N'\smallsetminus N\, \exists n<|\sigma'|\colon \sigma'(n)\subseteq\tau(i).\]
\end{definition}

\begin{lemma}[{\cite[Lem.~3.24]{CMR2}}]\label{c6}
Let $f\in\baire$ be increasing.
\begin{enumerate}[label=\normalfont (\arabic*)]
\item\label{c6:0}
The set\/ $\set{(\sigma,N,F)\in\Qor_f}{|F|=N}$
is dense.
In fact, if $\sigma'\supseteq\sigma$ and $F'\supseteq F$ are such that\/
$|\sigma'|\leq N^2$ and\/ $|F'|\leq N$,
then\/ $(\sigma',N,F')\geq(\sigma,N,F)$.

\item\label{c6:1}
For $n<\omega$ the set\/
$\set{(\sigma,N,F)\in\Qor_f}{n<N}$ is dense.
Even more, if\/ $(\sigma,N,F)\in\Qor_f$ and $N'\geq N$, then there is some $\sigma'$ such that\/ $(\sigma',N',F)\leq(\sigma,N,F)$.

\item
For $\tau\in(2^{<\omega})^\omega$, if\/ $\forall i<\omega$
$|\tau(i)|=f((i+1)^2)$, then the set\/
$\set{(\sigma,N,F)\in\Qor_f}{\tau\in F}$ is dense.

\item
For $n<\omega$ the set\/
$\set{(\sigma,N,F)\in\Qor_f}{n<|\sigma|}$ is dense.
\item
$\set{(\sigma, N, F)}{|\sigma|=N^2\text{ and\/ }|F|^2\leq N}$
is dense.
\end{enumerate}
\end{lemma}

The poset $\Qor_{f}$ is $\sigma$-$k$-linked for any $k<\omega$, so it has the ccc property. Let $G$ be a $\Qor_{f}$-generic filter over $V$. In $V[G]$, define \[\sigma_{\gen}:=\bigcup\set{\sigma}{\exists (N, F)\colon (\sigma, N, F) \in G}.\]

Using~\autoref{c6}, we can show the following using standard genericity arguments:

\begin{fact}
\
\begin{enumerate}[beginpenalty=10000,label=\rm(\arabic*)]
\item $\sigma_{\gen}\in(2^{<\omega})^\omega$.
\item $\hgt_{\sigma_\gen} = f$.
\item for every $\tau\in (2^{<\omega})^\omega\cap V$, if $|\tau(i)|\geq f((i+1)^2)$ for all but finitely many $i<\omega$, then $[\tau]_\infty\subseteq\bigcup_{n<\omega}[\dot\sigma_\gen(n)]$.
\end{enumerate}
\end{fact}

We below present a modification of a poset presented by Kamo and Osuga~\cite{KO} (see also~\cite{CM, BCM, BCM2}).

\begin{definition}[ED forcing]\label{c7}
Let $b=\seq{b(n)}{n<\omega}$ be a sequence of non-empty sets and $h\in\baire$ such that $\lim_{i\to\infty}\frac{h(i)}{|b(i)|}$ \sloppy $=0$ (when $b(i)$ is infinite, interpret $\frac{h(i)}{|b(i)|}$ as $0$).
Define the \emph{$(b,h)$-ED (eventually different real) forcing} $\Eor^h_b$ as the poset whose conditions are of the form $p=(s^p,\varphi^p)$ such that, for some $m:=m_{p}<\omega$,
\begin{enumerate}[label= \rm (\roman*)]
\item
$s^p\in\Seq_{<\omega}(b)$, $\varphi^p\in\Swf(b,m h)$, and
\item
$m h(i)<b(i)$ for every $i\geq|s^p|$,
\end{enumerate}
ordered by $(t,\psi)\leq(s,\varphi)$ iff $s\subseteq t$, $\forall\, i<\omega\colon \varphi(i)\subseteq\psi(i)$, and $t(i)\notin\varphi(i)$ for all $i\in|t|\menos|s|$.

When $G$ is $\Eor_{b}^{h}$-generic over $V$, we denote the generic real by $e_\gen:=\bigcup_{p\in G}s^p$, which we usually refer to as the \emph{eventually different generic real (over $V$)}. Denote $\Eor_{b}:=\Eor_b^1$ and $\Eor:=\Eor_\omega$. Thanks to~\autoref{c8}, we can use $\Eor_b$ to increase $\cov(\SNwf)$.
\end{definition}

It is widely recognized that $\Eor$ is $\sigma$-centered; however, the same does not hold for $\Eor^h_b$ when $b\in\baire$. Nonetheless, we have the following.

\begin{lemma}[{\cite[Lemma~2.21]{CM}}]\label{genlink}
Let $b,h\in\baire$ with $b\geq 1$. Let $\varrho,\rho\in\baire$ and assume that there is a non-decreasing function $f\in\baire$ going to infinity and an $m^*<\omega$ such that, for all but finitely many $k<\omega$,
\begin{enumerate}[label= \rm (\roman*)]
\item $k\varrho(k)h(i)<b(i)$ for all $i\geq f(k)$ and
\item $k\prod_{i=m^*}^{f(k)-1}((\min\{k,f(k)\}-1)h(i)+1)\leq\rho(k)$.
\end{enumerate}
Then\/ $\Eor^h_b$ is $(\rho,\varrho)$-linked.
\end{lemma}

\section{A quick revising of preservation theory}\label{sec:s3}

To make it easier for the reader, we will review the preservation properties introduced by Judah and the third author~\cite{JS} and Brendle~\cite{Br} for FS iterations of ccc posets, which were extended in~\cite[Sect.~4]{CM} by the first author and Mej\'ia. We examine new tools in~\cite{CarMej23,BCM2} for controlling the cardinal characteristics related to $\SNwf$ in forcing iterations. These characteristics will be utilized in demonstrating the consistency outcomes in~\autoref{sec:s5}

\begin{definition}\label{b1}
Let $\Rbf=\la X,Y,{\sqsubset}\ra$ be a relational system, and let $\theta$ be a cardinal.
\begin{enumerate}[label=\rm(\arabic*)]
\item Let $M$ be a~set.
\begin{enumerate}[label=\rm(\roman*)]
\item An object $y\in Y$ is \textit{$\Rbf$-dominating over $M$} if $x\sqsubset y$ for all $x\in X\cap M$.

\item An object $x\in X$ is \textit{$\Rbf$-unbounded over $M$} if it $\Rbf^\perp$-dominating over $M$, that is, $x\not\sqsubset y$ for all $y\in Y\cap M$.
\end{enumerate}

\item A family $\set{x_i}{i\in I}\subseteq X$ is \emph{strongly $\theta$-$R$-unbounded} if
$|I|\geq\theta$ and, for any $y\in Y$, $|\set{i\in I }{x_i\sqsubset y}|<\theta$.
\end{enumerate}
\end{definition}

In our forcing applications, we show that some cardinal characteristics have certain values (in a generic extension) by forcing a Tukey connection between their relational systems and some simple relational systems like $\Cbf_{[\lambda]^{<\theta}}$ and $[\lambda]^{<\theta}$ for some cardinals $\theta\leq\lambda$ with uncountable regular $\theta$.

For instance, if $\Rbf$ is a relational system and we force $\Rbf\eqT\Cbf_{[\lambda]^{<\theta}}$, then we obtain $\bfrak(\Rbf)=\non([\lambda]^{<\theta})=\theta$ and $\dfrak(\Rbf)=\cov([\lambda]^{<\theta})=\lambda$, the latter when either $\theta$ is regular or $\lambda>\theta$.
This discussion motivates the following characterizations of the Tukey order between $\Cbf_{[X]^{<\theta}}$ and other relational systems.

\begin{lemma}[{\cite[Lem.~1.16]{CM22}}]\label{b2}
Let\/ $\Rbf=\la X,Y,{\sqsubset}\ra$ be a relational system, $\theta$ be an infinite cardinal, and $I$ be a set of size ${\geq}\theta$.
\begin{enumerate}[label=\normalfont (\alph*)]
\item
$\Cbf_{[I]^{<\theta}}\leqT \Rbf$ iff there exists a strongly $\theta$-$\Rbf$-unbounded family\/ $\set{x_i}{i\in I}$.

\item
$\bfrak(\Rbf)\geq\theta$ iff\/ $\Rbf\leqT \Cbf_{[X]^{<\theta}}$.
\end{enumerate}
\end{lemma}

We look at the following types of well-defined relational systems.

\begin{definition}\label{b3}
Say that $\Rbf=\la X,Y,{\sqsubset}\ra$ is a \textit{Polish relational system (Prs)} if
\begin{enumerate}[label=\rm(\arabic*)]
\item\label{b3:1}
$X$ is a Perfect Polish space,
\item\label{b3:2}
$Y$ is a non-empty analytic subspace of some Polish $Z$, and
\item\label{b3:3}
${\sqsubset}=\bigcup_{n<\omega}{\sqsubset_{n}}$ where $\seq{{\sqsubset_{n}}}{n\in\omega}$ is some increasing sequence of closed subsets of $X\times Z$ such that, for any $n<\omega$ and for any $y\in Y$,
$({\sqsubset_{n}})^{y}=\set{x\in X}{x\sqsubset_{n}y}$ is closed nowhere dense.
\end{enumerate}
\end{definition}

\begin{remark}\label{b4}
By~\autoref{b3}~\ref{b3:3}, $\la X,\Mwf(X),{\in}\ra$ is Tukey below $\Rbf$ where $\Mwf(X)$ denotes the $\sigma$-ideal of meager subsets of $X$. Therefore, $\bfrak(\Rbf)\leq \non(\Mwf)$ and $\cov(\Mwf)\leq\dfrak(\Rbf)$.
\end{remark}

For the rest of this section, fix a Prs $\Rbf=\la X,Y,{\sqsubset}\ra$ and an infinite cardinal $\theta$.

\begin{definition}[Judah and Shelah {\cite{JS}}, Brendle~{\cite{Br}}]\label{b5}
A poset $\Por$ is \textit{$\theta$-$\Rbf$-good} if, for any $\Por$-name $\dot{h}$ for a member of $Y$, there is a nonempty set $H\subseteq Y$ (in the ground model) of size ${<}\theta$ such that, for any $x\in X$, if $x$ is $\Rbf$-unbounded over $H$ then $\Vdash x\not\sqsubset \dot{h}$.

We say that $\Por$ is \textit{$\Rbf$-good} if it is $\aleph_1$-$\Rbf$-good.
\end{definition}

The previous is a standard property associated with preserving $\bfrak(\Rbf)$ small and $\dfrak(\Rbf)$ large after forcing extensions.

\begin{remark}
Notice that $\theta<\theta_0$
implies that any $\theta$-$\Rbf$-good poset is $\theta_0$-$\Rbf$-good. Also, if $\Por \lessdot\Qor$ and $\Qor$ is $\theta$-$\Rbf$-good, then $\Por$ is $\theta$-$\Rbf$-good.
\end{remark}

\begin{lemma}[{\cite[Lemma~2.7]{CM}}]\label{b6}
Assume that $\theta$ is a regular cardinal. Then any poset of size ${<}\theta$
is $\theta$-$\Rbf$-good. In particular, Cohen forcing\/ $\Cor$ is\/ $\Rbf$-good.
\end{lemma}

\begin{example}\label{b0}
We now present the instances of Prs and the corresponding good posets we use in our applications.
\begin{enumerate}[label=\normalfont(\arabic*)]

\item\label{b0:1}
Define $\Omega_n:=\set{a\in [2^{<\omega}]^{<\aleph_0}}{\Lb(\bigcup_{s\in a}[s])\leq 2^{-n}}$ (endowed with the discrete topology) and put $\Omega:=\prod_{n<\omega}\Omega_n$ with the product topology, which is a perfect Polish space. For every $x\in \Omega$ denote
\[N_{x}:=\bigcap_{n<\omega}\bigcup_{m\geq n}\bigcup_{s\in x(m)}[s],\] which is clearly a Borel null set in $2^{\omega}$.

Define the Prs $\Cn:=\la\Omega,\cantor,{\sqsubset^{\rm n}}\ra$ where $x\sqsubset^{\rm n} z$ iff $z\notin N_{x}$. Recall that any null set in $\cantor$ is a subset of $N_{x}$ for some $x\in \Omega$, so $\Cn$ and $\Cbf_\Nwf^\perp$ are Tukey-Galois equivalent. Hence, $\bfrak(\Cn)=\cov(\Nwf)$ and $\dfrak(\Cn)=\non(\Nwf)$.

 Any $\mu$-centered poset is $\mu^+$-$\Cn$-good (\cite{Br}). In particular, $\sigma$-centered posets are $\Cn$-good.

\item\label{b0:2}
The relational system $\Ed_b$ is Polish when $b=\seq{b(n)}{n<\omega}$ is a sequence of non-empty countable sets such that $|b(n)|\geq 2$ for infinitely many $n$.
Consider $\Ed:=\la\baire,\baire,{\neq^\infty}\ra$.
By~\cite[Thm.~2.4.1 \& Thm.~2.4.7]{BJ} (see also~\cite[Thm.~5.3]{CMlocalc}), $\bfrak(\Ed)=\non(\Mwf)$ and $\dfrak(\Ed)=\cov(\Mwf)$.

\item\label{b0:3} The relational system $\baire:=\la\baire,\baire,{\leq^*}\ra$ is Polish. Typical examples of $\baire$-good sets are $\Eor^h_b$, $\Qor_f$ and random forcing. More generally, $\sigma$-$\Fr$-linked posets are $\D$-good (see~\cite{mejiavert,BCM,CMR2}).

\item\label{b0:4}
For each $k<\omega$, let $\id^k:\omega\to\omega$ such that $\id^k(i)=i^k$ for all $i<\omega$ and $\Hcal:=\largeset{\id^{k+1}}{k<\omega}$. Let $\Lc^*:=\la\baire, \Scal(\omega, \Hcal),{\in^*}\ra$ be the Polish relational system where \[\Swf(\omega, \Hcal):=\set{\varphi\colon \omega\to[\omega]^{<\aleph_0}}{\exists{h\in\Hcal}\, \forall{i<\omega}\colon|\varphi(i)|\leq h(i)},\]
and recall that $x\in^*\varphi$ iff $\forall^\infty n\colon x(n)\in\varphi(n)$. As a consequence of~\cite[Thm.~2.3.9]{BJ} (see also~\cite[Thm.~4.2]{CMlocalc}), $\bfrak(\Lc^*)=\add(\Nwf)$ and $\dfrak(\Lc^*)=\cof(\Nwf)$.

Any $\mu$-centered poset is $\mu^+$-$\Lc^*$-good (see~\cite{Br,JS}) so, in particular, $\sigma$-centered posets are $\Lc^*$-good. Besides, Kamburelis~\cite{Ka} showed that any Boolean algebra with a strictly positive finitely additive measure is $\Lc^*$-good (in particular, any subalgebra of random forcing).

\item\label{b0:5} Let $\Mbf := \la\cantor,\Ior\times\cantor,{\sqsubm}\ra$ where
\[x \sqsubm (I,y) \text{ iff }\forall^\infty n\colon x\frestr I_n \neq y\frestr I_n.\]
This is a Polish relational system and $\Mbf\eqT \Cbf_\Mwf$ (by Talagrand~\cite{Tal98}, see e.g.~\cite[Prop.~13]{BWS}).

Note that, whenever $M$ is a transitive model of $\thzfc$, $c\in\cantor$ is a Cohen real over~$M$ iff $c$ is $\Mbf$-unbounded over~$M$.

\item\label{b0:6} In~\cite[Sec.~5]{BCM2}, we present a Polish relation system $\Rbf^f_\Gwf$, parametrized by a~countable set $\{f\}\cup\Gwf$ of increasing functions in $\baire$, which is useful to control $\add(\SNwf)$ and $\cof(\SNwf)$ in FS iterations (see \autoref{b9}). We do not need to review the definition of this relational system, but it is enough to indicate that any (poset forcing equivalent to a) Boolean algebra with a strictly positive finitely additive measure, and any $\sigma$-centered poset, are $\Rbf^f_\Gwf$-good (\cite[Thm.~5.8 \& Cor.~5.9]{BCM2}).
\end{enumerate}
\end{example}

\begin{example}[{\cite[Ex.~4.19]{CM}}]\label{KOpre}
Kamo and Osuga~{\cite{KO}} define a gPrs with parameters $\varrho,\rho\in\baire$, which we denote by $\aLc^*(\varrho,\rho)$. For the purposes of this paper, it is just enough to review its properties. Assume that $\varrho>0$ and $\rho\geq^* 1$.
\begin{enumerate}[label = \rm (\alph*)]
    \item\label{KOa} $\aLc^*(\varrho,\rho)\leqT \aLc(\varrho,\rho^{\id})$~\cite[Lem.~4.21]{CM}.
    \item\label{KOb} If $\sum_{i<\omega}\frac{\rho(i)^i}{\varrho(i)}<\infty$ then $\aLc(\varrho,\rho^{\id})\leqT\Cbf_{\Nwf}^\perp$~\cite[Lem.~2.3]{KM21}, so $\cov(\Nwf)\leq\bfrak(\aLc^*(\varrho,\rho))$ and $\dfrak(\aLc^*(\varrho,\rho))\leq\non(\Nwf)$
    \item\label{KOc} If $\varrho\not\leq^*1$ and $\rho\geq^*1$, then any $(\rho,\varrho^{\rho^{\id}})$-linked poset is $\aLc^*(\varrho,\rho)$-good (see~\cite[Lem.~10]{KO} and~\cite[Lem.~4.23]{CM}).
    \item\label{KOd} Any $\theta$-centered poset is $\theta^+$-$\aLc^*(\varrho,\rho)$-good~\cite[Lem.~4.24]{CM}.
\end{enumerate}
\end{example}

We now turn to FS (finite support) iterations.

\begin{definition}[Direct limit]\label{b7}
We say that $\seq{\Por_i}{i\in S}$ is a \emph{directed system of posets} if $S$ is a directed preorder and, for any $j\in S$, $\Por_j$ is a poset and $\Por_i\subsetdot\Por_j$ for all $i\leq_S j$.

For such a system, we define its \emph{direct limit} $\limdir_{i\in S}\Por_i:=\bigcup_{i\in S}\Por_i$ ordered by
\[q\leq p \sii \exists\, i\in S\colon p,q\in\Por_i\text{ and }q\leq_{\Por_i} p.\]
\end{definition}

The Cohen reals added along an iteration are usually used as witnesses for Tukey connections, as they form strong witnesses. For example:

\begin{lemma}[{\cite[Lemma~4.14]{CM}}]\label{lem:strongCohen}
Let $\mu$ be a cardinal with uncountable cofinality, let\/ $\seq{\Por_{\alpha}}{\alpha<\mu}$ be a\/ $\subsetdot$-increasing sequence of\/ $\cf(\mu)$-cc posets and let\/ $\Por_\mu=\limdir_{\alpha<\mu}\Por_{\alpha}$. If\/ $\Por_{\alpha+1}$ adds a Cohen real $\dot{c}_\alpha\in X$ over $V^{\Por_\alpha}$ for any $\alpha<\mu$, then\/ $\Por_{\mu}$ forces that\/ $\set{\dot{c}_\alpha}{\alpha<\mu}$ is a $\mu$-$\Rbf$-unbounded family. In particular, $\Por_\mu$ forces that $\mu\leqT\Cbf_\Mwf \leqT \Rbf$.
\end{lemma}

Good posets are preserved along FS iterations as follows.

\begin{theorem}[{\cite[Sec.~4]{BCM2}}]\label{b8}
Let\/ $\seq{ \Por_\xi,\Qnm_\xi}{\xi<\pi}$ be a FS iteration such that, for $\xi<\pi$, $\Por_\xi$ forces that\/ $\Qnm_\xi$ is a non-trivial $\theta$-cc $\theta$-$\Rbf$-good poset.
Let\/ $\set{\gamma_\alpha}{\alpha<\delta}$ be an increasing enumeration of\/ $0$ and all limit ordinals smaller than $\pi$ (note that $\gamma_\alpha=\omega\alpha$), and for $\alpha<\delta$ let $\dot c_\alpha$ be a\/~$\Por_{\gamma_{\alpha+1}}$-name of a Cohen real in $X$ over $V_{\gamma_\alpha}$.

Then\/ $\Por_\pi$ is $\theta$-$\Rbf$-good. Moreover,
if $\pi\geq\theta$ then\/ $\Cbf_{[\pi]^{<\theta}}\leqT\Rbf$, $\bfrak(\Rbf)\leq\theta$ and\/ $|\pi|\leq\dfrak(\Rbf)$.
\end{theorem}

We even have nice theorems for $\SNwf$.

\begin{theorem}[{\cite[Thm.~5.10]{BCM2}}]\label{b9}
Let $\theta_0\leq \theta$ be uncountable regular cardinals,
let $\lambda$ be a~cardinal such that $\lambda=\lambda^{<\theta_0}$
 and let $\pi=\lambda\delta$ (ordinal product) for some ordinal\/ $0<\delta<\lambda^+$. Assume $\theta \leq \lambda$ and $\cf(\pi)\geq \theta_0$. If $\Por$ is a~FS iteration of length $\pi$ of non-trivial $\theta_0$-cc $\theta$-$\Rbf^f_\Gwf$-good posets of size ${\leq}\lambda$,
then $\Por$ forces\/ $\Cbf_{[\lambda]^{<\theta}}\leqT\SNwf$, in particular,
$\add(\SNwf)\leq\theta$ and $\lambda\leq\cof(\SNwf)$.
\end{theorem}

We now present two preservation results for the covering of $\SNwf$, originally introduced by Pawlikowski~\cite{P90} and generalized and improved in~\cite{CarMej23}. Here, we use the notion of the \emph{segment cofinality} of an ordinal $\pi$:
\[\scf(\pi):=\min\set{|c|}{c\subseteq \pi \text{ is a non-empty final segment of }\pi}.\]

\begin{theorem}[{\cite{P90},~\cite[Thm.~5.4~(c)]{CarMej23}}]\label{b10}
Let\/ $\seq{ \Por_\xi}{\xi\leq\pi}$ be a $\subsetdot$-increasing sequence of posets such that\/ $\Por_\pi=\limdir_{\xi<\pi}\Por_\xi$. Assume that\/ $\cf(\pi)>\omega$, $\Por_\pi$ has the\/ $\cf(\pi)$-cc and\/ $\Por_{\xi+1}$ adds a Cohen real over the\/ $\Por_\xi$-generic extension for all $\xi<\pi$. Then $\pi\leqT \Cbf_{\SNwf}^\perp$, in particular\/ $\cov(\SNwf) \leq \cf(\pi) \leq \non(\SNwf)$.
\end{theorem}

\begin{theorem}[{\cite{P90},~\cite[Cor.~5.9]{CarMej23}}]\label{b11}
Assume that $\theta\geq\aleph_1$ is regular. Let\/
$\Por_\pi=\seq{\Por_\xi,\Qnm_\xi}{ \xi<\pi}$ be a FS iteration of non-trivial precaliber $\theta$ posets such that\/ $\cf(\pi)>\omega$ and\/ $\Por_\pi$ has\/ $\cf(\pi)$-cc,
and let $\lambda:=\scf(\pi)$. Then\/ $\Por_\pi$ forces\/ $\Cbf_{[\lambda]^{<\theta}}\leqT\Cbf_{\SNwf}^{\perp}$. In particular, whenever\/ $\scf(\pi)\geq\theta$, $\Por_\pi$ forces\/ $\cov(\SNwf)\leq\theta$ and\/ $\scf(\pi)\leq\non(\SNwf)$.
\end{theorem}

To force a lower bound of $\bfrak(\Rbf)$, we use:

\begin{theorem}[{\cite[Thm.~2.12]{CM22}}]\label{b12}
Let $\theta\geq\aleph_1$ be a regular cardinal, and let\/ $\Por_\pi=\seq{\Por_\xi,\Qnm_\xi}{\xi<\pi}$ be a FS iteration of $\theta$-cc posets with\/ $\cf(\pi)\geq\theta$. Assume that, for all $\xi<\pi$ and any $A\in[X]^{<\theta}\cap V_\xi$, there is some $\eta\geq\xi$ such that\/ $\Qnm_\eta$ adds an $R$-dominating real over~$A$. Then\/ $\Por_\pi$ forces $\theta\leq\bfrak(\Rbf)$, i.e.\ $\Rbf\leqT\Cbf_{[X]^{<\theta}}$.
\end{theorem}

\begin{lemma}[{\cite[Lemma~4.5]{CM}}]\label{b13}
Assume that $\theta$ has uncountable cofinality. Let\/ $\seq{\Por_{\alpha}}{\alpha<\theta}$ be a\/ $\subsetdot$-increasing sequence of\/ $\cf(\theta)$-cc posets such that\/ $\Por_\theta=\limdir_{\alpha<\theta}\Por_{\alpha}$. If\/ $\Por_{\alpha+1}$ adds a Cohen real $\dot{c}_\alpha\in X$ over $V^{\Por_\alpha}$ for any $\alpha<\theta$, then\/ $\Por_{\theta}$ forces that\/ $\set{\dot{c}_\alpha}{\alpha<\theta}$ is a~strongly $\theta$-$\Rbf$-unbounded family, i.e.\ $\theta\leqT \Rbf$.
\end{lemma}

\section{Forcing the additivity of \texorpdfstring{$\Nwf$}{} small}\label{sec:s4}

We present a property that helps us to keep the additivity of the null ideal small under certain forcing extensions. Indeed, we present a novel kind of linkedness property that falls between being $\sigma$-centered and $\sigma$-linked, which is not as strong as $(\rho, \varrho)$-linked that helps us to keep this cardinal small under generic extensions.

The next definition gives a new type of ccc forcing notions not increasing the additivity of the null ideal.

\begin{definition}
\label{d1}
Let $\Por$ be a~forcing notion and let $\bar\rho=\seq{\rho_n}{n\in\omega}$ be a sequence of functions $\rho_n\colon\omega\to\omega$ such that $\lim_{k\to\infty}\rho_n(k)=\infty$, $\rho_n(k)\leq k$, $\rho_n(k+1)\geq2$ and $\rho_n\geq\rho_{n+1}$.
\begin{enumerate}[label=\rm(\arabic*)]
\item\label{d1:1}
Say that $\Por$ is \emph{$\sigma$-$\bar\rho$-linked} if there is a sequence $\set{P_n}{n\in\omega}$ such that:\smallskip
\begin{enumerate}[label=\rm(\alph*)]
\item\label{d1:a}
$\bigcup_{n\in\omega}P_n$ is dense in $\Por$ and
\item\label{d1:b}
for all $n$, $k<\omega$, if $\set{p_i}{i<k}\subseteq P_n$, then there is subset $Q$ of $\set{p_i}{i<k}$ of size $\rho_n(k)$ that has a common upper bound.
\end{enumerate}

\item
Say that $\Por$ is \emph{$\sigma$-$\bar\rho{\star}$-linked},
if there is a sequence $\seq{P_n}{n\in\omega}$ such that:\smallskip
\begin{enumerate}[label=\rm(\alph*)]
\item $\Por=\bigcup_{n\in\omega}P_n$ and
\item $\forall n,k\in\omega\
\forall\seq{p_i}{i<k}\in{}^kP_n\
\exists u\in[k]^{\rho_n(k)}$ the set
$\set{p_i}{i\in u}$ has an upper bound.
\end{enumerate}
\item We say that $\Por$ is \emph{$\sigma{\uparrow}$-$\bar\rho$-linked} or \emph{$\sigma{\uparrow}$-$\bar\rho{\star}$-linked}, respectively,
if moreover, $P_n\subseteq P_{n+1}$ for all $n\in\omega$.
\end{enumerate}
\end{definition}

The following result indicates the relationship between the prior properties.

\begin{fact}
\
\begin{enumerate}[beginpenalty=10000,label=\rm(\arabic*)]
    \item Any $\sigma$-$\bar\rho{\star}$-linked poset is $\sigma$-$\bar\rho$-linked.
    \item Any $\sigma{\uparrow}$-$\bar\rho$-linked poset is $\sigma$-$\bar\rho$-linked.
\end{enumerate}
\end{fact}

We now present some examples. Think of random forcing $\Bor$ as the family of Borel sets
$A\subseteq\cantor$ of positive measure ordered by inclusion.
The set
\[
\set{A\in\Bor}
{\forall s\in{}^{<\omega}2\bsp\colon A\cap[s]\ne\emptyset\Rightarrow\mu(A\cap[s])>0}
\]
is a~dense subset of~$\Bor$ because
\[\Lb(A)=\Lb(A\smallsetminus\bigcup\set{[s]}{s\in{}^{<\omega}2\setand\Lb(A\cap[s])=0}).\]

The random perfect set forcing is the set
\[
\PBor=
\set{(n,A)\in\omega\times\Bor}
{\forall s\in{}^{<\omega}2\colon A\cap[s]\ne\emptyset\Rightarrow\Lb(A\cap[s])>0}
\]
ordered by $(n',A')\geq(n,A)$ if $n'\geq n$, $A'\subseteq A$, and
$A'\frestr n=A\frestr n$ (where $A\frestr n=\set{x\frestr n}{x\in A}$).
Clearly, the forcing~$\PBor$ adds a~perfect set of random reals.

\begin{lemma}\label{example-barrho}
\
\begin{enumerate}[beginpenalty=10000,label=\rm(\arabic*)]
\item\label{exa:1}
$\sigma$-centered posets are $\sigma$-$\bar\rho$-linked for $\rho_n(k)=k$.
\item\label{exa:2}
Random forcing\/~$\Bor$ is $\sigma$-$\bar\rho$-linked for some~$\bar\rho$.
\item\label{exa:3}
Random perfect set forcing\/ $\PBor$ is $\sigma$-$\bar\rho$-linked
for some~$\bar\rho$.
\end{enumerate}
\end{lemma}

\begin{proof}
We just prove~\ref{exa:2} and~\ref{exa:3} because~\ref{exa:1} is obvious.

\ref{exa:2}: Let $\seq{c_n}{n\in\omega}$ be any sequence of positive reals
with\/ $\lim_{n\to\infty}c_n=\infty$ and let
$\rho_n(k)=m$, if $(m-1)c_n<k\leq mc_n$ and $m,n,k\in\omega$.
Denote
$P_n=\set{A\in\Bor}{\Lb(A)\geq1/c_n}$.
Then $\Bor=\bigcup_{n\in\omega}P_n$.
It is enough to prove the next claim.

\begin{clm}\label{rho-random}
If\/ $\set{A_i}{i<k}\subseteq P_n$, then there is
$a\in[k]^{\rho_n(k)}$ with\/ $\Lb(\bigcap_{i\in a}A_i)>0$.
\end{clm}

Assume that $m\geq1$,
$\seq{A_i}{i<k}$ is a~sequence of Borel sets in~$\cantor$ of measure~$1/c_n$,
and for every $b\in[k]^{m+1}$, $\Lb(\bigcap_{i\in a}A_i)=0$.
For $a\subseteq k$ and $j\leq k$ denote
$x_a=\Lb(\set{x\in\cantor}{\set{i<k}{x\in A_i}=a})$
and $z_j=\sum_{a\in[k]^j}x_a$.
Then
\begin{itemize}
\item
$\sum_{j\leq m}z_j=\sum_{a\in[k]^{\leq m}}x_a=1$,
\item
$\forall i<k$
$\sum\set{x_a}{i\in a\in[k]^{\leq m}}=1/c_n$ and consequently,
\item
$k/c_n=\sum_{j=1}^mj\sum_{a\in[k]^j}x_a=\sum_{j\leq m}jz_j\leq\sum_{j\leq m}mz_j=m$.
\end{itemize}
Therefore, $k\leq mc_n$.
This finishes the proof of the claim.

\ref{exa:3}: Let $\seq{c_n}{n\in\omega}$ and $\bar\rho=\seq{\rho_n}{n\in\omega}$ be the same as
in the proof of~(2).
Define $\bar\rho^*=\seq{\rho^*_n}{n\in\omega}$ by
$\rho^*_n(k)=\rho_n{}^{2^n}(\lfloor k/2^{2^n+1}\rfloor)$ where
for a~function~$f$, $f^0(x)=x$ and $f^{n+1}(x)=f(f^n(x))$.
We show that $\PBor$~is $\sigma$-$\bar\rho^*$-linked.

For $p=(n,A)\in\PBor$ we write $n_p=m$, $A_p=A$, and
$S_p=\set{s\in{}^{n_p}2}{A_p\cap[s]\ne\emptyset}$.
Clearly $\PBor=\bigcup_{n\in\omega}Q_n$ where
\[
Q_n=\set{p\in\PBor}{n_p\leq n\setand\forall s\in S_p\bsp\mu(A_p\cap[s])\geq1/c_n}.
\]

Assume $\set{p_i}{i<k}\subseteq Q_n$ then there is
$S\in\bigcup_{m\leq n}\Pwf({}^m2)$ such that
the set $\set{i<k}{S_{p_i}=S}$ has cardinality~$\geq k'$ where
$k'=\lfloor k/2^{2^n+1}\rfloor$.
Since $|S|\leq2^n$, by applying~\autoref{rho-random} for every $s\in S$
we find $b\subseteq a$ of cardinality $\rho_n{}^{2^n}(k')=\rho^*_n(k)$ such that the set
$\set{p_i}{i\in b}$ has a~common extension in~$\PBor$.

The proof of~\autoref{example-barrho} is complete.
\end{proof}

Our linkedness property is related to the intersection number, which was originally introduced by Kelley~\cite{Kelley59} (see also~\cite[Sec.~4]{CMU} and~\cite{U23}).

\begin{definition}
Let $\Por$ be a forcing notion and $Q\subseteq\Por$.
\begin{enumerate}
\item
For a finite sequence $\bar q=\seq{q_i}{i<n}\in{}^n\Por$, we define
\[
i_\ast^\Por(\bar q)=\max
\set{|F|}{F\subseteq n\wedge\set{q_i}{i\in F}
\text{ has a lower bound in $\Por$}}.
\]

\item
The \emph{intersection number of $Q$ in $\Por$}, denoted by
$\Int^\Por(Q)$, is defined by
\[
\Int^\Por(Q)=\inf\largeset
{\frac{i_\ast^\Por(\bar q)}{n}}
{\bar q\in{}^nQ\wedge n\in\omega\smallsetminus\{0\}}.
\]
\end{enumerate}
\end{definition}

\begin{lemma}\label{d14}
Assume that\/ $\Por=\bigcup_{n}Q_n$ where\/ $\Int^\Por(Q_n)\geq\frac{1}{n}$ for each $n$. Then\/ $\Por$ is $\sigma$-$\bar\rho$-linked where $\rho_n(k)=\frac{k}{n}$.
\end{lemma}
\begin{proof}
It suffices to prove that for all $n$, $k<\omega$, if $\set{p_i}{i<k}\subseteq Q_n$, then there is subset $Q$ of $\set{q_i}{i<k}$ of size $\rho_n(k)$ that has a common upper bound. Assume that $\set{p_i}{i<k}\subseteq Q_n$. Then $\frac{i_\ast^\Por(\bar q)}{k}\geq\frac{1}{n}$, so $i_\ast^\Por(\bar q)\geq\frac{k}{n}$. Therefore, $\set{q_i}{i<\rho_n(k)}$ has a lower bound.
\end{proof}

The following shows that poset $\Qor_f$ is $\sigma$-$\bar\rho$-linked.

\begin{lemma}\label{Qf-barrho}
$\Qor_f$ is $\sigma$-$\bar\rho$-linked for some $\bar\rho$.
\end{lemma}

\begin{proof}
Fix $n\in\omega$ and $\sigma$ such that $(\sigma,n,F)\in\Qor_f$ for some~$F$ and define
\begin{align*}
&Q_n:=\set{p\in\Qor_f}{N_p=|F_p|=n},\\
&Q_{n,\sigma}:=\set{p\in Q_n}{\sigma_p=\sigma}.
\end{align*}
By~\autoref{c6}~\ref{c6:0} the set
$\bigcup_{n\in\omega}Q_n$ is dense in~$\Qor_f$.

\begin{clm}\label{d5}
Let $m\geq 2$, $N\geq m\cdot n$, and $A\in[Q_{n,\sigma}]^m$ be such
that there is a~finite set\/~$\Phi$ such that for every $p\in A$, $F_p\frestr[n,N)=\Phi$.
Then the set $A$ has a~common extension in\/~$\Qor_f$ of the form\/
$(\sigma',N,F)$ with $F=\bigcup_{p\in A}F_p$.
\end{clm}

\begin{proof}
By the assumptions we can choose an enumeration
$\set{\varphi_j}{j<n}$ of the set $F\frestr[n,N)$ consisting of functions with domain $[n,N)$ such that $\varphi_j(i)\in{}^{f((i+1)^2)}2$ for $i\in[n,N)$.
For every $k<n(N-n)$ define $\psi_k=\varphi_j(i)$, if $k=n(i-n)+j$
where $j<n$ and $i\in[n,N)$.
Define $\sigma'\supseteq\sigma$ such that $|\sigma'|=n^2+n(N-n)$
(recall that $|\sigma|\leq n^2$) and $\sigma'(n^2+k)=\psi_k\frestr f(n^2+k)$ for $k<n(N-n)$.
We claim that $(\sigma',N,F)\in\Qor_f$.

Clearly, $|F|\leq m\cdot n\leq N$ and $|\sigma'|=n^2+n(N-n)=nN\leq N^2$.
The third item in defining conditions in $\Qor_f$ is
fulfilled.

We verify the first item in the definition of $\Qor_f$.
We need to prove that for every $k<n(N-n)$, $f(n^2+k)\leq|\psi_k|$,
i.e., for every $j<n$ and every $i\in[n,N)$,
$f(n^2+n(i-n)+j)\leq f((i+1)^2)$.
Since $f$ is increasing it is enough to do this for $j=n-1$ and
it is enough to verify that $n^2+n(i-n)+(n-1)\leq(i+1)^2$.
Since $i\geq n$, $n^2+n(i-n)+(n-1)=ni+n-1\leq i^2+i\leq(i+1)^2$.

By definition of $\sigma'$ it is easy to see that $(\sigma',N,F)$
extends every $p\in A$.
\end{proof}
We continue with the proof of~\autoref{Qf-barrho}.
The number $k_n$ of all $\sigma$'s in conditions from~$Q_n$
is~$\leq\prod_{i\leq n^2}2^{f(i)}$
because
$\sum_{j\leq n^2}\prod_{i<j}2^{f(i)}=
\sum_{j\leq n^2}2^{\sum_{i<j}f(i)}\leq
2^{\sum_{i\leq n^2}f(i)}=\prod_{i\leq n^2}2^{f(i)}$.

The number $c_{m,n}$ of $n$~element sets of partial functions
defined on $[n,m\cdot n)$ and satisfying the third item in
definition of $\Qor_f$ is less or equal to the combinatorial number
is $C(\prod_{i\in[n,m\cdot n)}2^{f((i+1)^2)},n)$.
\begin{enumerate}
\item[(1)]
If $A\in[Q_n]^{(l-1)\cdot k_n+1}$, then there is a~set
$C\in[A]^l$ with the same~$\sigma_p$ for all $p\in C$.
\item[(2)]
If $C\in[Q_{n,\sigma}]^{(m-1)\cdot c_{n,m}+1}$, then there is a~set
$B\in[Q_{n,\sigma}]^m$ such that all $p\in B$ have the same
$F_p\frestr[n,m\cdot n)$.
By~\autoref{d5}, the conditions in~$B$ have a~common extension.
\end{enumerate}
Taking $l=(m-1)\cdot c_{n,m}+1$ in (1) and then applying~(2)
we get
\begin{enumerate}
\item[(3)]
If $A\in[Q_n]^{(m-1)\cdot c_{n,m}\cdot k_n+1}$, then there is
$B\in[A]^m$ such that all $p\in B$ have the same $\sigma_p$ and
the same $F_p\frestr[n,m\cdot n)$
and hence, the conditions in~$B$ have a~common extension.
\end{enumerate}
Define $\rho_n(k)=m$ if $(m-1)\cdot c_{n,m}\cdot k_n+1\leq k<m\cdot c_{n,m}\cdot k_n+1$.
Since $\bigcup_{n\in\omega}Q_n$ is dense in~$\Qor_f$,
by~(3) it follows that $\Qor_f$ is $\sigma$-$\bar\rho$-linked.
\end{proof}

Below, we proceed to show that $\Eor_b^h$ is $\sigma$-$\bar\rho$-linked.

\begin{lemma}\label{d8}
Let $b,h\in\baire$ be such that $b\geq 1$ and\/
$\lim_{i\to\infty}h(i)/b(i)=0$.
Then\/ $\Eor^h_b$~is $\sigma$-$\bar\rho$-linked for some
sequence of functions $\bar\rho=\seq{\rho_n}{n\in\omega}$ such that
for every $n\in\omega$, $\lim_{k\to\infty}\rho_n(k)=\infty$.
\end{lemma}

The proof of~\autoref{d8} is a consequence of the following lemma:

\begin{lemma}\label{d11}
Let\/ $\Por$ be a~forcing notion and let\/
$\seq{Q_m}{m\in\omega}$,
$\seq{r_m}{m\in\omega}$,
$\seq{\pi_{m,j}}{m,j\in\omega}$,
and\/
$\seq{a_{m,j}}{m,j\in\omega}$ be sequences with the following properties:
\begin{enumerate}[label=\rm(\arabic*)]
\item
$\bigcup_{m\in\omega}Q_m$ is a~dense subset of\/~$\Por$,
\item
for every $m,j\in\omega$,
$r_m\colon\omega\to\omega$, $\pi_{m,j}\colon Q_m\to a_{m,j}$, $|a_{m,j}|<\omega$,
for every $x\in a_{m,j}$, $\pi_{m,j}^{-1}[\{x\}]$ is
an $r_m(j)$-linked set of cardinality~$\geq r_m(j)$, and
\item
for every $m\in\omega$, $\lim_{j\to\infty}r_m(j)=\infty$.
\end{enumerate}
Then\/ $\Por$ is $\sigma$-$\bar\rho$-linked for some
$\bar\rho=\seq{\rho_n}{n\in\omega}$ such that
for every $n\in\omega$, $\lim_{i\to\infty}\rho_n(i)=\infty$.
\end{lemma}

\begin{proof}
Define $P_n=\bigcup\set{Q_m}{m\leq n}$,
$r^{\min}_n(j)=\min\set{r_m(j)}{m\leq n}$, and
$v_n(j)=(n+1)\cdot\max\set{|a_{m,j}|\cdot r_m(j)}{m\leq n}$.
By the pigeonhole principle and by~(ii), for every $A\in[P_n]^{v_n(j)}$
there is a~set $B\in[A]^{r^{\min}_n(j)}$ with a~common extension.
Define
\[
\rho_n(i)=r^{\min}_n(j),\quad
\text{if $v_n(j)\leq i<v_n(j+1)$.}
\]
Now for every $n,i<\omega$ and $A\in[P_n]^i$ there is $B\in[A]^{\rho_n(i)}$ such that $B$ has a~common extension.
\end{proof}

\begin{proof}[Proof of~\autoref{d8}]
We verify the assumptions of~\autoref{d11} with index~$m$ replaced
by ordered pairs $(m,k)$.
For $m,k\in\omega$ let
$Q_{m,k}=\set{p\in\Eor^h_b}
{s^p=k\setand\forall i\geq k\bsp|\varphi^p(i)|\leq m\cdot h(i)<b(i)}$.
It is easy to see that the set $\bigcup_{m,k\in\omega}Q_{m,k}$
is a~dense subset of~$\Eor^h_b$.
For $m,k,j\in\omega$ denote
$a_{m,k,j}=\Seq_{\leq k}(b)\times(\prod_{i<k}\Pwf(b(i))\times
\prod_{i\in[k,j)}[b(i)]^{\leq m\cdot h(i)})$
and define
$\pi_{m,k,j}\colon Q_{m,k}\to a_{m,k,j}$ and $r_{m,k}\colon\omega\to\omega$
by
\begin{align*}
&\pi_{m,k,j}(p)=(s^p,\varphi^p\frestr\max\{k,j\}),\\
&r_{m,k}(j)=\max\set{n\in\omega}{\forall i\geq\max\{k,j\}\bsp n\cdot m\cdot h(i)<b(i)}.
\end{align*}
It is easy to see that $\lim_{j\to\infty}r_{m,k}(j)=\infty$
and for every $x\in a_{m,k,j}$, the set $\pi_{m,k,j}^{-1}[\{x\}]$ is infinite
and $r_{m,k}(j)$-linked.
\end{proof}

We conclude this section by presenting the following property to keep $\add(\Nwf)$ small in forcing extensions.
This one is useful for posets that satisfy some linkedness as in~\autoref{d1}.

\begin{definition}\label{d7}
Let $\bar\rho=\seq{\rho_n}{n\in\omega}$ as in~\autoref{d1}.
For $h\in\baire$ let

\begin{enumerate}
\item
$g_m(k)=\max\set{n\in\omega}{\rho_m(n)\leq h(k)}$ and $h'(k)=\sum_{m\le k}g_m(k)$,
\item
$\Hwf_{\bar\rho,h}=\set{h_n}{n\in\omega}\subseteq\baire$ where $h_0=h$ and $h_{n+1}=(h_n)'$.
\end{enumerate}
Note that $g_m(k)\geq h(k)$ because $2\leq\rho_m(n)\le n$
(for $n\geq 2$) and hence $h'(k)\geq(k+1)\cdot h(k)$.
Therefore $h_{n+1}(k)\geq(k+1)^{n+1}$ whenever $h(k)\geq1$.
\end{definition}

\begin{fact}
Let $\bar\rho=\seq{\rho_n}{n\in\omega}$ as in~\autoref{d1}. Then $\Lc^*_{\Hwf_{\bar\rho,h}}\eqT\Nwf$ whenever
$h$~is positive.
\end{fact}

Inspired by~\cite[Lem.~5.14]{BrM}, we present the following lemma.

\begin{lemma}\label{d2}
Let $\bar\rho$ be as in~\autoref{d1}, let $h\in\baire$, and let
$h'$~be as in \autoref{d7}.
If\/ $\Por$ is $\sigma$-$\bar\rho$-linked and $\dot\varphi$ is
a\/~$\Por$-name for a~member in $\Swf(\omega,h)$, then there is
a $\varphi'\in\Swf(\omega,h')$ such that for any $f\in\baire$,
$f\not\in^*\varphi'$ implies\/ $\Vdash f\not\in^*\dot\varphi$.
\end{lemma}

\begin{proof}
Let $\Por$ be a $\sigma$-$\bar\rho$-linked poset witnessed by $\seq{Q_m}{m\in\omega}$. For each $m\in\omega$ define $\varphi_m,\varphi'\colon\omega\to[\omega]^{<\aleph_0}$ by
\[
\varphi_m(k)=
\set{i\in\omega}{\exists q\in Q_m\,\colon q\Vdash i\in\dot\varphi(k)}
\quad\text{and}\quad
\varphi'(k)=\bigcup_{m\le k}\varphi_m(k)
\]

\begin{clm}
$|\varphi_m(k)|\leq g_m(k)$ for all $k$.
\end{clm}

\begin{proof}
For every $i\in\varphi_m(k)$ let $q_i\in Q_m$ be such that
$q_i\Vdash i\in\dot\varphi(k)$.
Denote $n=|\varphi_m(k)|$ and let $a\in[\varphi_m(k)]^{\rho_m(n)}$
be such that the set $\set{q_i}{i\in a}$ has an upper bound~$r$.
Then $\rho_m(n)\leq h(k)$ because $r$~forces that $a\subseteq\dot\varphi(k)$.
By definition of~$g_m(k)$, $|\varphi_m(k)|=n\leq g_m(k)$.
(In fact, $\sigma$-$\bar\rho{\star}$-linked is used.)
\end{proof}
Then
$\forall k$ $|\varphi'(k)|\leq h'(k)$ and
$\forall k\geq m$
$\varphi_m(k)\subseteq\bigcup_{i\le k}\varphi_i(k)=\varphi'(k)$.
\begin{clm}
If $f\in\baire$ with $f\not\in^*\varphi'$, then $\Vdash f\not\in^*\dot\varphi$.
\end{clm}

\begin{proof}
 Assume that $f\in\baire$ and $f\not\in^*\varphi'$.
 Towards a~contradiction assume that there are $p\in\Por$ and
 $k_0\in\omega$ such that
 $p\Vdash``\forall k\geq k_0$ $f(k)\in\dot\varphi(k)$''.
 Let $m\in\omega$ be such that $p\in Q_m$.
 We can assume that $k_0\geq m$.
 Then $\forall k\geq k_0$ $f(k)\in\varphi_m(k)\subseteq\varphi'(k)$
 which is a~contradiction.
\end{proof}
This completes the proof of the lemma.
\end{proof}

As a direct consequence of~\autoref{d2}, we infer:

\begin{corollary}\label{d3}
Let $\bar\rho$ be as in~\autoref{d1}. Assume
that $h\in\baire$ converges to infinity. Then there is some $\Hwf_{\bar\rho,h}=\set{h_n}{n\in\omega}\subseteq\baire$ with $h_0=h$ such that any $\sigma$-$\bar\rho$-linked posets is\/ $\Lc^*(\Hwf_{\bar\rho,h})$-good.
\end{corollary}

\section{Applications to the left side of Cichon's diagram}\label{sec:s5}

This section provides applications of our results in forcing iterations. In particular, we perform forcing constructions by using the uf-extendable matrix iteration technique from~\cite{BCM} defined below, to prove~\autoref{Thm:a0},~\ref{Thm:a2}, and~\ref{Thm:a1}.

We now present the matrix iterations with ultrafilters method, which helps to force many simultaneous values in Cicho\'n's diagram.

\begin{definition}[{\cite[Def.~2.10]{BCM}}]\label{f1}
A~\emph{simple matrix iteration} of ccc posets (see~\autoref{matrixuf}) is composed of the following objects:
\begin{enumerate}[label=\rm (\Roman*)]
\item ordinals $\gamma$ (height) and $\pi$ (length);
\item a function $\Delta\colon \pi\to\gamma$;
\item a sequence of posets $\seq{\Por_{\alpha,\xi}}{\alpha\leq \gamma,\ \xi\leq \pi}$ where $\Por_{\alpha,0}$ is the trivial poset for any $\alpha\leq \gamma$;
\item for each $\xi<\pi$, $\dot{\Qor}^*_\xi$ is a
$\Por_{\Delta(\xi),\xi}$-name of a poset such that $\Por_{\gamma,\xi}$ forces it to be ccc;
\item $\Por_{\alpha,\xi+1}=\Por_{\alpha,\xi}\ast\Qnm_{\alpha,\xi}$, where
\[\dot{\mathbb{Q}}_{\alpha,\xi}:=
\begin{cases}
\Qnm^*_\xi & \textrm{if $\alpha\geq\Delta(\xi)$,}\\
\{0\}& \textrm{otherwise;}
\end{cases}\]
\item for $\xi$ limit, $\Por_{\alpha,\xi}:=\limdir_{\eta<\xi}\Por_{\alpha,\eta}$.
\end{enumerate}

It is known that $\alpha\leq\beta\leq\gamma$ and $\xi\leq\eta\leq\pi$ imply $\Por_{\alpha,\xi}\subsetdot\Por_{\beta,\eta}$, see e.g.~\cite{B1S} and \cite[Cor.~4.31]{CM}. If $G$ is $\Por_{\gamma,\pi}$-generic
over $V$, we denote $V_{\alpha,\xi}= [G\cap\Por_{\alpha,\xi}]$ for all $\alpha\leq\gamma$ and $\xi\leq\pi$.
\end{definition}

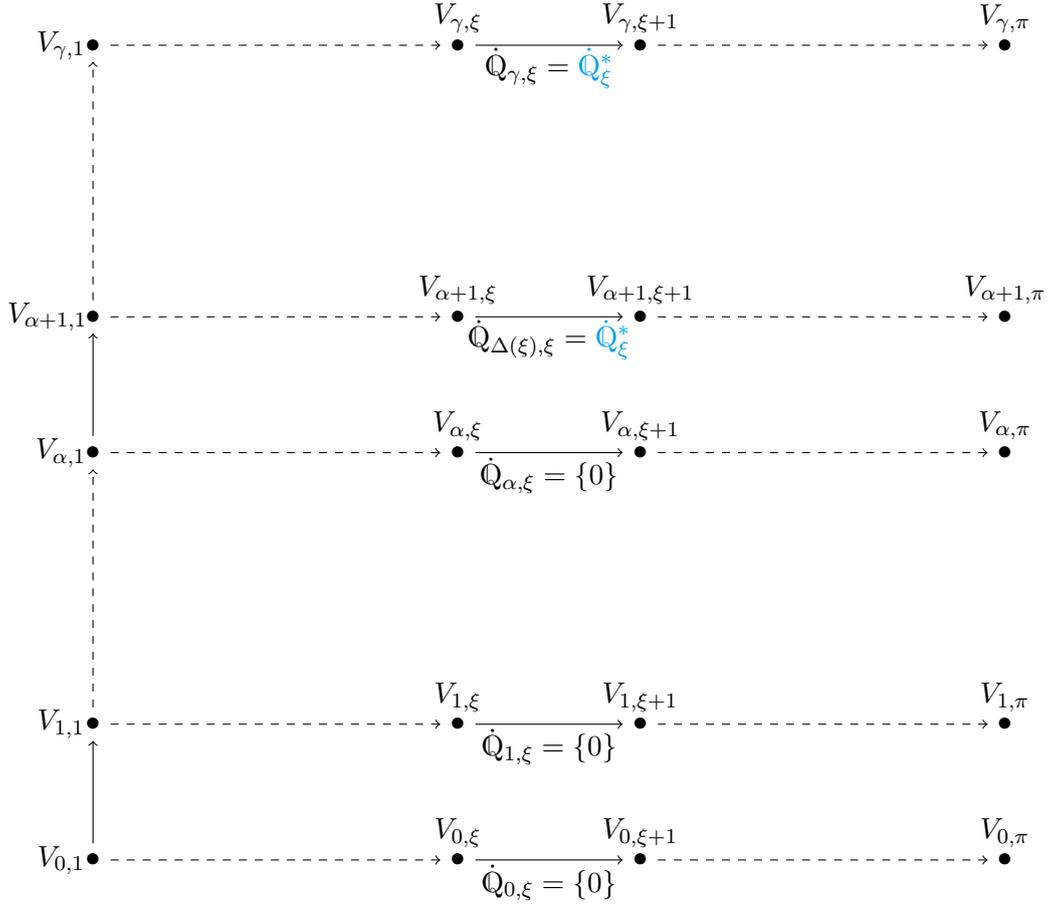
\begin{figure}[ht]
\centering
\begin{tikzpicture}[scale=1.2]
\small{

\node (f00) at (0,0){$\bullet$};
\node (f01) at (0,1.5){$\bullet$};
\node (f02) at (0,4.5) {$\bullet$} ;
\node (f03) at (0,6) {$\bullet$} ;
\node (f04) at (0,9) {$\bullet$} ;
\node (fxi0) at (4,0){$\bullet$};
\node (fxi1) at (4,1.5){$\bullet$};
\node (fxi2) at (4,4.5) {$\bullet$};
\node (fxi3) at (4,6) {$\bullet$} ;
\node (fxi4) at (4,9) {$\bullet$} ;
\node (fxi+10) at (6,0){$\bullet$};
\node (fxi+11) at (6,1.5){$\bullet$};
\node (fxi+12) at (6,4.5) {$\bullet$};
\node (fxi+13) at (6,6) {$\bullet$} ;
\node (fxi+14) at (6,9) {$\bullet$} ;
\node (fpi0) at (10,0){$\bullet$};
\node (fpi1) at (10,1.5){$\bullet$};
\node (fpi2) at (10,4.5) {$\bullet$};
\node (fpi3) at (10,6) {$\bullet$} ;
\node (fpi4) at (10,9) {$\bullet$} ;

\draw   (f00) edge [->] (f01);
\draw[dashed]    (f01) edge [->] (f02);
\draw       (f02) edge [->] (f03);
\draw[dashed]       (f03) edge [->] (f04);

\draw[dashed]   (f00) edge [->] (fxi0);
\draw      (fxi0) edge [->] (fxi+10);
\draw[dashed]       (fxi+10) edge [->] (fpi0);

\draw[dashed]   (f01) edge [->] (fxi1);
\draw        (fxi1) edge [->] (fxi+11);
\draw[dashed]        (fxi+11) edge [->] (fpi1);

\draw[dashed]  (f02) edge [->] (fxi2);
\draw        (fxi2) edge [->] (fxi+12);
\draw[dashed]        (fxi+12) edge [->] (fpi2);

\draw[dashed]  (f03) edge [->] (fxi3);
\draw       (fxi3) edge [->] (fxi+13);
\draw[dashed]      (fxi+13) edge [->] (fpi3);

\draw[dashed]  (f04) edge [->] (fxi4);
\draw        (fxi4) edge [->] (fxi+14);
\draw[dashed]        (fxi+14) edge [->] (fpi4);

\node at (-0.35,0) {$V_{0,1}$};
\node at (-0.35,1.5) {$V_{1,1}$};
\node at (-0.35,4.5) {$V_{\alpha,1}$};
\node at (-0.5,6) {$V_{\alpha+1,1}$};
\node at (-0.35,9) {$V_{\gamma,1}$};

\node at (4,0.3) {$V_{0,\xi}$};
\node at (6,0.3) {$V_{0,\xi+1}$};

\node at (4,1.8) {$V_{1,\xi}$};
\node at (6,1.8) {$V_{1,\xi+1}$};

\node at (4,4.8) {$V_{\alpha,\xi}$};
\node at (6,4.8) {$V_{\alpha,\xi+1}$};

\node at (4,6.3) {$V_{\alpha+1,\xi}$};
\node at (6,6.3) {$V_{\alpha+1,\xi+1}$};

\node at (4,9.3) {$V_{\gamma,\xi}$};
\node at (6,9.3) {$V_{\gamma,\xi+1}$};

\node at (10,0.3) {$V_{0,\pi}$};
\node at (10,1.8) {$V_{1,\pi}$};
\node at (10,4.8) {$V_{\alpha,\pi}$};
\node at (10,6.3) {$V_{\alpha+1,\pi}$};
\node at (10,9.3) {$V_{\gamma,\pi}$};

\node at (5,-0.25) {$\Qnm_{0,\xi}=\{0\}$};
\node at (5,1.25) {$\Qnm_{1,\xi}=\{0\}$};
\node at (5,4.25) {$\Qnm_{\alpha,\xi}=\{0\}$};
\node at (5,5.75) {$\Qnm_{\Delta(\xi),\xi}=\subiii{\Qnm^*_\xi}$};
\node at (5,8.75) {$\Qnm_{\gamma,\xi}=\subiii{\Qnm^*_\xi}$};



}
\end{tikzpicture}
\caption{A simple matrix iteration}
\label{matrixuf}
\end{figure}

\begin{lemma}[{\cite[Lemma~5]{BrF}, see also~\cite[Cor.~2.6]{mejiavert}}]\label{realint}
 Assume that\/ $\Por_{\gamma, \pi}$ is a simple matrix iteration as in~\autoref{f1} with\/ $\cf(\gamma)>\omega$.
Then, for any $\xi\leq\pi$,
\begin{enumerate}[label=\rm (\alph*)]
\item
$\Por_{\gamma,\xi}$ is the direct limit of $\seq{\Por_{\alpha,\xi}}{\alpha<\gamma}$, and
\item
if $\eta<\cf(\gamma)$ and $\dot{f}$ is a\/ $\Por_{\gamma,\xi}$-name of a function from $\eta$ into\/ $\bigcup_{\alpha<\gamma}V_{\alpha,\xi}$ then $\dot{f}$ is forced to be equal to a\/ $\Por_{\alpha,\xi}$-name for some $\alpha<\gamma$. In particular, the reals in $V_{\gamma,\xi}$ are precisely the reals in\/ $\bigcup_{\alpha<\gamma}V_{\alpha,\xi}$.
\end{enumerate}
\end{lemma}

Using a Polish relational system that is Tukey-equivalent with $\Cbf_\Mwf$ (see \autoref{b0}~\ref{b0:5}) we have the following result.

\begin{theorem}[{\cite[Thm.~5.4]{CM}}]\label{matsizebd}
Let\/ $\Por_{\gamma, \pi}$ be a simple matrix iteration as in~\autoref{f1}. Assume that, for any $\alpha<\gamma$, there is some $\xi_\alpha<\pi$ such that\/ $\Por_{\alpha+1,\xi_\alpha}$ adds a~Cohen real $\dot{c}_\alpha\in X$ over~$V_{\alpha,\xi_\alpha}$.
Then, for any $\alpha<\gamma$, $\Por_{\alpha+1,\pi}$ forces that $\dot{c}_{\alpha}$ is Cohen over~$V_{\alpha,\pi}$.

In addition, if\/ $\cf(\gamma)>\omega_1$ and $f\colon \cf(\gamma)\to\gamma$ is increasing and cofinal, then\/ $\Por_{\gamma,\pi}$ forces that\/ $\set{\dot{c}_{f(\zeta)}}{\zeta<\cf(\gamma)}$ is a strongly\/ $\cf(\gamma)$-$\Cbf_\Mwf$-unbounded family. In particular, $\Por_{\gamma,\pi}$~forces $\gamma\leqT \Cbf_\Mwf$ and\/ $\non(\Mwf)\leq\cf(\gamma)\leq\cov(\Mwf)$.
\end{theorem}

 Mej\'ia~\cite{mejiavert} introduced the notion of \emph{ultrafilter-linkedness} (abbreviated uf-linkedness).
He proved that no $\sigma$-uf-linked poset adds dominating reals and that such a poset preserves a certain type of mad (maximal almost disjoint) families. 
These results where improved in~\cite{BCM}, which motivated the construction of matrix iterations of ${<}\theta$-uf-linked posets. 

The following notion formalizes the matrix iterations with ultrafilters from~\cite{BCM}. The property ``${<}\theta$-uf-linked" is used as a black box, i.e.\ there is no need to review its definition, but it is enough to present the relevant examples and facts (with proper citation).

\begin{definition}[{\cite[Def.~4.2]{BCM}}]\label{f4}
Let $\theta\geq\aleph_1$ and let $\Por_{\gamma, \pi}$ be a simple matrix iteration as in~\autoref{f1}. Say that $\Por_{\gamma, \pi}$ is a~\emph{${<}\theta$-uf-extendable matrix iteration} if for each $\xi<\pi$, $\Por_{\Delta(\xi),\xi}$ forces that $\Qnm_\xi$ is a ${<}\theta$-uf-linked poset.
\end{definition}

\begin{example}\label{exm:ufl}
The following are the instances of ${<}\theta$-uf-linked posets that we use in our applications.
\begin{enumerate}[label=\rm (\arabic*)]

\item Any poset of size $\mu<\theta$ is ${<}\theta$-uf-linked. In particular, Cohen
forcing is $\sigma$-uf-linked (i.e.\ ${<}\aleph_1$-uf-linked), see~\cite[Rem.~3.3~(5)]{BCM}.

\item Random forcing is $\sigma$-uf-linked~\cite[Lem.~3.29 \& Lem.~5.5]{mejiamatrix}.

\item Let $b$ and $h$ be as in~\autoref{c7}. Then $\Eor_b^h$ is $\sigma$-uf-linked~\cite[Lem.~3.8]{BCM}.

\item $\Qor_f$ is $\sigma$-uf-linked~\cite[Thm.~3.25]{CMR2}.
\end{enumerate}
\end{example}

\begin{theorem}[{\cite[Thm.~4.4]{BCM}}]\label{mainpres}
Assume that $\theta\leq\mu$ are uncountable cardinals with $\theta$~regular. Let\/ $\Por_{\gamma,\pi}$ be a ${<}\theta$-uf-extendable matrix iteration as in~\autoref{f4} such that
\begin{enumerate}[label=\rm (\roman*)]
\item $\gamma\geq\mu$ and $\pi\geq\mu$,
\item for each $\alpha<\mu$, $\Delta(\alpha)=\alpha+1$ and $\Qnm_\alpha$ is Cohen forcing, and
\item $\dot c_\alpha$ is a $\Por_{\alpha+1,\alpha+1}$-name of the Cohen real in $\baire$ added by\/ $\Qnm_\alpha$.
\end{enumerate}
Then\/ $\Por_{\alpha,\pi}$ forces that\/ $\set{\dot c_\alpha}{\alpha<\mu}$ is $\theta$-$\baire$-unbounded, in particular, $\Cbf_{[\mu]^{<\theta}}\leqT \baire$.
\end{theorem}

Prior to proving our main results, we review some other posets that we will use.

\begin{definition}\label{f0} Define the following forcing notions:
\begin{enumerate}[label=\normalfont(\arabic*)]
\item \emph{Localization forcing} is the poset defined by $\Loc:=\set{(n,\varphi)\in\omega\times \Swf(\omega,\id_\omega)}{\exists m<\omega\,\forall i<\omega\colon |\varphi(i)|\leq m}$ ordered by $(n',\varphi')\leq(n,\varphi)$ iff $n\leq n'$, $\varphi'\frestr n=\varphi\frestr n$ and $\varphi(i)\subseteq\varphi'(i)$ for every $i<\omega$. This forcing is used to increase $\add(\Nwf)$. Recall that $\Loc$ is $\sigma$-linked, hence ccc.

\item\label{defposet2} \emph{Hechler forcing} is defined by $\Dor=\omega^{<\omega}\times\baire$,
ordered by $(t,g)\leq(s,f)$ if $s\subseteq t$, $f\leq g$ and $f(i)\leq t(i)$ for all $i\in |t|\menos|s|$. This forcing is used to increase $\bfrak$. Recall that $\Dor$ is $\sigma$-centered.


\item $\Cor_\lambda:=\Fn_{<\aleph_0}(\lambda\times\omega,2)$ is the poset adding $\lambda$-many Cohen reals.
\end{enumerate}
\end{definition}

We are ready to prove~\autoref{Thm:a0}.

\begin{theorem}\label{appl:I}
Let $\theta_1\leq\theta_2\leq\theta_3\leq\theta_4\leq\theta_5\leq \theta_6\leq\theta_7=\theta_7^{\aleph_0}$ be uncountable regular cardinals and assume that\/ $\cof([\theta_7]^{<\theta_i})=\theta_7$ for $1 \leq i \leq 5$.
Then there is a~ccc poset forcing:
\begin{enumerate}[label=\rm (\alph*)]
\item\label{e0:a}
$\cfrak=\theta_7$;
\item\label{e0:b} $\Cbf_{[\theta_7]^{<{\theta_1}}}\leqT\Lc^*$, $\SNwf\leqT\Cbf_{[\theta_7]^{<{\theta_2}}}$,  $\Cbf_{[\theta_7]^{<{\theta_3}}}\leqT\Cbf_{\Nwf}^{\perp}$, $\Cbf_{[\theta_7]^{<{\theta_4}}}\leqT\Cbf_{\SNwf}^{\perp}$, and\/ $\Cbf_{[\theta_7]^{<{\theta_5}}}\leqT\baire$;

\item\label{e0:c} $\Lc^*\leqT\Cbf_{[\theta_7]^{<{\theta_1}}}$, $\Cbf_{[\theta_7]^{<{\theta_2}}}\leqT\SNwf$,  $\Cbf_{\Nwf}^{\perp}\leqT\Cbf_{[\theta_7]^{<{\theta_3}}}$, $\Cbf_{\SNwf}^{\perp}\leqT\Cbf_{[\theta_7]^{<{\theta_4}}}$,  and $\baire\leqT\Cbf_{[\theta_7]^{<{\theta_4}}}$;
\item\label{e0:d}
$\theta_6\leqT\Cbf_\Mwf$, $\theta_7\leqT \Cbf_\Mwf$, and\/ $\Ed\leqT \theta_7\times\theta_6$;
\item\label{e0:e}
$\theta_7^{\theta_7}\leqT\SNwf$.
\end{enumerate}
In particular, it forces
\begin{multline*}
\add(\Nwf)=\theta_1\leq\add(\SNwf)=\theta_2\leq\cov(\Nwf)=
\theta_3\leq\cov(\SNwf)= \supcov=\theta_4 \leq\\
\leq \bfrak=\theta_5
\leq\non(\Mwf)=\theta_6
\leq\cov(\Mwf)=\non(\SNwf)=\dfrak=\cfrak=\theta_7 <\dfrak_{\theta_7} \leq \cof(\SNwf).
\end{multline*}
\end{theorem}
\begin{proof}
For each $\rho<\theta_7\theta_6$ denote $\eta_\rho:=\theta_7\rho$. Fix a bijection $g=(g_0, g_1,g_2):\theta_7\to\{0, 1,2,3,4\}\times\theta_7\times\theta_7$ and fix a function $t\colon\theta_7\theta_6\to\theta_7$ such that, for any $\alpha<\theta_7$, the set $\set{\rho<\theta_7\theta_6}{t(\rho)=\alpha}$ is cofinal in $\theta_7\theta_6$.

We construct a ${<}\theta_5$-uf-extendable matrix iteration $\Por$ as follows: First, construct (by recursion) increasing functions $\rho, \varrho, b \in\baire$ such that, for all $k<\omega$,
\begin{enumerate}[label=(\alph*)]
\item $k^{k+1} \leq \rho(k)$,
\item $\sum_{i<\omega}\frac{\rho(i)^i}{\varrho(i)}<\infty$ and
\item $k \varrho(k)^{\rho(k)^k} < b(k)$;
\end{enumerate}
It follows by~\autoref{genlink} that $\Eor_b=\Eor^1_b$ is $(\rho, \varrho^{\rho^\id})$-linked and by~\autoref{d8} that $\Eor_b$ is $\sigma$-$\bar\rho$-linked. On the other hand, for some increasing function $f\in\baire$, by employing~\autoref{Qf-barrho}, it follows that $\Qor_f$ is $\sigma$-$\bar\rho$-linked

We now construct the ${<}\theta_5$-uf-extendable matrix iteration $\Por_{\gamma,\pi}$ with $\gamma=\theta_7$ and $\pi=\theta_7\theta_7\theta_6$. First set,
\begin{enumerate}[label=\rm (C\arabic*)]
\item\label{cohbas} $\Delta(\alpha):=\alpha+1$ and $\Qnm_\alpha$ is Cohen forcing for $\alpha<\theta_7$.
\end{enumerate}
Define the matrix iteration at each $\xi=\eta_\rho+\varepsilon$ for $0<\rho<\theta_7\theta_6$ and $\varepsilon<\theta_7$ as follows. Denote\footnote{We think of $X_1$ as the set of Borel codes of Borel sets with measure zero.}
\[(\Qnm_j^*,X_j):=\begin{cases}
(\Loc,\baire) & \text{if $j=0$,} \\
(\Qor_{f},\SNwf) & \text{if $j=1$,}\\
(\Bor,\Bwf(\cantor)\cap \Nwf) & \text{if $j=2$,}\\
(\Eor_b,\prod b) & \text{if $j=3$,}\\
(\Dor,\baire) & \text{if $j=4$.}
\end{cases}\]
For $j<5$, $0<\rho<\theta_7\theta_6$ and $\alpha<\theta_7$, choose
\begin{enumerate}[label=(E$j$)]
\item\label{Ej} a collection $\set{\Qnm_{j,\alpha,\zeta}^\rho}{\zeta<\theta_7}$ of nice $\Por_{\alpha,\eta_\rho}$-names for posets of the form~$(\Qor^*_j)^N$ for some transitive model $N$ of ZFC with $|N|<\theta_{j+1}$
such that, for any $\Por_{\alpha,\eta_\rho}$-name $\dot F$ of a subset of $X_j$ of size ${<}\theta_{j+1}$, there is some $\zeta<\theta_7$ such that, in $V_{\alpha,\eta_\rho}$, $\Qnm^\rho_{j,\alpha,\zeta}=(\Qor^*_j)^N$ for some $N$ containing $\dot F$ (we explain later why this is possible),
\end{enumerate}
and set
\begin{enumerate}[label=\rm (C\arabic*)]
\setcounter{enumi}{1}
\item
$\Delta(\xi):=t(\rho)$ and $\Qnm_{\xi}:=\Eor^{V_{\Delta(\xi),\xi}}$ when $\xi=\eta_\rho$;

\item
$\Delta(\xi):=g_1(\varepsilon)$ and $\Qnm_{\xi}:=\Qnm^{\rho}_{g(\varepsilon)}$ when $\xi=\eta_\rho+1+\varepsilon$ for some $\varepsilon< \theta_7$.
\end{enumerate}
Clearly, $\Por:=\Por_{\theta_7,\pi}$ satisfies the ccc property. We can now show that $\Por$ forces what we want. Keep in mind that $\Por$ can be obtained by the FS iteration $\seq{\Por_{\theta_7,\xi},\Qnm_{\theta_7,\xi}}{\xi<\pi}$. It should be clear that $\Por$ forces $\cfrak=\theta_7$, so~\ref{e0:a} is done.

Item~\ref{e0:b}: We start by proving that $\Por$ forces $\Cbf_{[\theta_7]^{<\theta_1}}\leqT\Lc^*$, as the rest is proved in a~similar way. For this
argument we interpret $\Por$ as the iteration $\seq{\Por_{\theta_7,\xi},\Qnm_{\theta_7,\xi}}{\xi<\pi}$, so it suffices to argue $\Por$ forces $\Cbf_{[\theta_7]^{<\theta_1}}\leqT\Lc^*$ thanks~\autoref{b8}. In order to achieve this, it is enough to verify that, for each $\xi<\pi$, $\Por_{\gamma, \xi}$ forces that $\Qnm_{\gamma, \xi}$ is $\theta_1$-$\Lc^*$-good:
\begin{itemize}
\item The cases $\xi<\theta_7$ and $\xi=\eta_\rho$ for $\rho>0$ follow by~\autoref{b0}~\ref{b0:6};
\item when $\xi=\eta_\rho+1+\varepsilon$ for some $\rho>0$ and $\varepsilon<\theta_7$, we split into four subcases:
\begin{itemize}
\item the case $g_0(\varepsilon)=0$ is clear by~\autoref{b6};
\item when $g_0(\varepsilon)=1$ it follows by~\autoref{d3};
\item when $g_0(\varepsilon)=2$, it follows by~\autoref{b0}~\ref{b0:4};
\item when $g_0(\varepsilon)=3$, use~\autoref{d3}; and
\item when $g_0(\varepsilon)=4$, it follows by~\autoref{b0}~\ref{b0:4};
\end{itemize}
\end{itemize}

To get that $\Por$ forces $\Cbf_{[\theta_7]^{<{\theta_2}}}\leqT\SNwf$: It suffices to verify as in the previous case that all iterands are $\lambda_1$-$\Rbf^f_\Gwf$-good (see \autoref{b0}~\ref{b0:6}), so by~\autoref{b9}  $\Cbf_{[\theta_7]^{<{\theta_2}}}\leqT\SNwf$ is forced.

On the other hand, since $\Por$ can be obtained by the FS iteration $\seq{\Por_{\theta_7,\xi},\Qnm_{\theta_7,\xi}}{\xi<\pi}$ and
all its iterands are $\theta_3$-$\aLc^*(\varrho,\rho)$-good (see~\autoref{KOpre}~\ref{KOb}),  $\Por$ forces $\Cbf_{[\theta_6]^{<{\theta_3}}}\leqT \aLc^*(\varrho,\rho^\id) \leqT \Cbf_{\Nwf}^{\perp}$ by applying~\autoref{b8}. To see that $\Por$ forces $\Cbf_{[\theta_7]^{<{\theta_4}}}\leqT\Cbf_{\SNwf}^{\perp}$ the argument is even simpler: This follows by~\autoref{b11} since $\Por$ is obtained by the FS iteration of precaliber $\theta_4$ posets.

Accordingly with~\autoref{exm:ufl}, the matrix iteration is ${<}\theta_5$-uf-extendable. Consequently, using~\autoref{mainpres}, $\Por$ forces $\Cbf_{[\theta_7]^{<{\theta_5}}}\leqT\baire$.

Item~\ref{e0:c}: We now prove that $\Por$ forces $\Lc^*\leqT\Cbf_{[\theta_7]^{<{\theta_1}}}$. This is also basically the same
to what is presented in~\cite[Thm.~6.9]{BCM2}.  Let $\dot A$ be a $\Por$-name for a subset of $\baire$ of size~${<}\theta_2$. By employing~\autoref{realint} we can can find $\alpha<\theta_7$ and $\rho<\theta_7\theta_6$ such
that $\dot A$ is $\Por_{\alpha,\eta_\rho}$-name. By~(E0), we can find a $\zeta<\theta_6$ and a $\Por_{\alpha,\eta_\rho}$-name $\dot N$ of a transitive model
of ZFC of size~${<}\theta_2$ such that $\Por_{\alpha,\eta_\rho}$ forces that $\dot N$ contains $\dot A$ as a subset and $\Loc^{\dot N}=\Qnm_{0,\alpha,\zeta}^\rho$, so the
generic slalom added by $\Qnm_{\xi}=\Qnm_{g(\varepsilon)}^\rho$ localizes all the reals in $\dot A$ where $\varepsilon:=g^{-1}(0,\alpha,\zeta)$ and $\xi=\eta_\rho+1+\varepsilon$.

The remaining of the statements of Item~\ref{e0:c} is basically proved as above argument. Item~\ref{e0:d} is exactly as proved in~\cite[Thm.~6.9]{BCM2}. However, we present their proof to keep the document self-contained. Since $\cf(\pi)=\theta_5$, by using~\autoref{lem:strongCohen}, $\Por$ forces that $\theta_5\leqT\Cbf_\Mwf$ and, by~\autoref{matsizebd}, $\Por$~forces $\theta_6\leqT\Cbf_\Mwf$. We are left with seeing that $\Por$~forces that $\Ed\leqT \theta_6\times\theta_6\theta_5$ (because $\theta_6\theta_5\eqT \theta_5$). For this purpose, for each $\rho<\theta_6\theta_5$ denote by $\dot e_\rho$ the $\Por_{\Delta(\eta_\rho),\eta_\rho+1}$-name of the eventually different real over $V_{t(\rho),\eta_\rho}$ added by $\Qnm_{t(\rho),\eta_\rho}$. In $V_{\gamma, \pi}$, we are going to define maps $\Psi_-:\baire\to\theta_6\times\theta_6\theta_5$ and $\Psi_+:\theta_6\times\theta_6\theta_5\to\baire$ such that, for any $x\in\baire$ and for any $(\alpha,\rho)\in\theta_6\times\theta_6\theta_5$, if $\Psi_-(x)\leq(\alpha,\rho)$ then $x\neq^\infty\Psi_+(\alpha,\rho)$.

For $x\in V_{\theta_6,\pi}\cap\baire$, we can find $\alpha_x<\theta_6$ and $\rho_x<\theta_6\theta_5$ such that $x\in V_{\alpha_x,\eta_{\rho_x}}$, so put $\Psi_-(x):=(\alpha_x,\rho_x)$; for $(\alpha, \rho)\in\theta_6\times\theta_6\theta_5$, find some $\rho'<\theta_6\theta_5$ such that $\rho'\geq\rho$ and $t(\rho')=\alpha$, and define $\Psi_+(\alpha,\rho):=\dot e_{\rho'}$. It is clear that $(\Psi_-,\Psi_+)$ is
the required Tukey connection.

On the other hand, since $\Por$ forces $\cov(\Mwf)=\non(\SNwf)=\dfrak=\cfrak=\theta_7$, $\Por$ forces $\theta_7^{\theta_7}\leqT\SNwf$ by using~\cite[Cor.~4.25]{CarMej23}, so item~\ref{e0:e} is done.
\end{proof}

The prior result can be modified to ensure that $\cov(\Mwf)<\dfrak$, however, $\non(\SNwf)<\cof(\SNwf)$ cannot be attained. 

\begin{theorem}\label{appl:II}
Let $\theta_1\leq\theta_2\leq\theta_3\leq\theta_4 \leq\theta_5\leq \theta_6\leq\theta_7$ be uncountable regular cardinals, and $\theta_8=\theta_8^{\aleph_0}\geq \theta_7$ satisfying\/ $\cof([\theta_8]^{<\theta_i}) = \theta_8$ for\/ $1\leq i\leq 5$.
    Then there is a~ccc poset forcing:
\begin{enumerate}[label=\rm (\alph*)]
\item\label{e1:a} $\cfrak = \theta_8$;
\item \label{e1:c} $\Lc^*\eqT\Cbf_{[\theta_7]^{<{\theta_1}}}$, $\SNwf\eqT\Cbf_{[\theta_7]^{<{\theta_2}}}$,  $\Cbf_{\Nwf}^{\perp}\eqT\Cbf_{[\theta_7]^{<{\theta_3}}}$, $\Cbf_{\SNwf}^{\perp}\eqT\Cbf_{[\theta_7]^{<{\theta_4}}}$,  and $\baire\eqT\Cbf_{[\theta_7]^{<{\theta_4}}}$;
\item \label{e1:d} $\theta_6\leqT\Cbf_\Mwf$, $\theta_7\leqT \Cbf_\Mwf$, and\/ $\Ed\leqT \theta_7\times\theta_6$; and
\item\label{e1:f} if\/ $\cf(\theta_8) = \theta_6$ then\/ $\theta_6^{\theta_6} \leqT \SNwf$.
\end{enumerate}
In particular, it forces
\begin{align*}
\add(\Nwf)=\theta_1\leq\add(\SNwf)=\theta_2\leq\cov(\Nwf)=\theta_3\leq \cov(\SNwf)= \supcov = \theta_4 \leq\\
\leq \bfrak=\theta_5
\leq\non(\Mwf)=\theta_6
\leq\cov(\Mwf)=\theta_7 \leq \non(\SNwf)=\dfrak=\cfrak=\theta_8.
\end{align*}
\end{theorem}

\begin{proof}
We proceed as in~\autoref{appl:I}, built a ${<}\theta_5$-uf-extendable matrix iteration $\Por$, but use $\eta_\rho:=\theta_8 \rho$ for $\rho<\theta_7\theta_6$, a bijection $g\colon \theta_8\to \{0,1,2,3,4\}\times \theta_7\times \theta_8$, and a function $t\colon \theta_7\theta_6\to \theta_7$ such that $\set{\rho<\theta_7\theta_6}{t(\rho)=\alpha}$ is cofinal in $\theta_7\theta_6$ for all $\alpha<\theta_7$. Then, $\Por$~forces~\ref{e1:a}--\ref{e1:d}.

To show Item~\ref{e1:f}, we use notions and results from~\cite{CarMej23} that we do not fully review. In this reference, we define a principle we denote by $\mathbf{DS}(\delta)$, with the parameter an ordinal $\delta$, which has a profound effect on $\cof(\SNwf)$. According to~\cite[Cor.~6.2]{CarMej23}, $\Por$~forces $\mathbf{DS}(\theta_8\theta_6)$. On the other hand,~\cite[Thm.~4.24]{CarMej23} states that, if $\mathbf{DS}(\delta)$ holds and $\non(\SNwf)=\supcof$ has the same cofinality as $\delta$, say $\theta$, then $\la\theta^\theta,{\leq}\ra\leqT\SNwf$. Since $\Por$~forces $\non(\SNwf)=\supcof=\theta_8$ and $\cf(\theta_8) = \cf(\theta_8 \theta_6) = \theta_6$, we conclude~\ref{e1:f}.
\end{proof}

We now wrap up this section by showing~\autoref{Thm:a1}.

\begin{theorem}\label{appl:III}
Let $\theta_1\leq\theta_2\leq\theta_3\leq\theta_4$ be uncountable regular cardinals, and $\theta_5$ a cardinal such that $\theta_5\geq\theta_4$ and\/ $\cof([\theta_5]^{<\theta_i})=\theta_4$ for $i=1,2$ and $\theta_5 = \theta_5^{\aleph_0}$. Then there is some ccc poset forcing:
\begin{enumerate}[label=\rm(\arabic*)]
    \item\label{SN1} $\cfrak=\theta_5$;

    \item $\Lc^*\eqT \Cbf_{[\theta_5]^{<{\theta_1}}}$;

    \item\label{SN3} $\baire \eqT \Cbf_{[\theta_5]^{<{\theta_2}}}$;

    \item\label{SN4} $\theta_4\leqT\Cbf_\Mwf$ and $\theta_5\leqT \Cbf_\Mwf$;

    \item\label{SN5} $\theta_4\leqT\Cbf^\perp_\SNwf$ and $\theta_5\leqT \Cbf^\perp_\SNwf$;

    \item\label{SN6} $\SNwf \leqT (\theta_4\times \theta_4)^{\theta_5}$; and

    \item\label{SN7} $\Cbf^\perp_\Nwf \leqT \theta_5\times\theta_4$.
\end{enumerate}
In particular, it is forced that:
\begin{align*}
\add(\Nwf)=\theta_1 &\leq\bfrak=\theta_2\leq
\add(\SNwf)= \cov(\SNwf)= \cov(\Nwf) = \non(\Mwf)=\theta_3\\
 & \leq \cov(\Mwf)= \non(\SNwf) = \non(\Nwf) = \theta_4\leq\dfrak=\cfrak=\theta_5.
\end{align*}
\end{theorem}
\begin{proof}
First add $\theta_4$-many Cohen reals, and afterwards make a ${<}\theta_2$-uf-extendable matrix iteration $\Por = \Por_{\gamma,\pi}$ with $\gamma = \lambda_4$ and $\pi = \theta_4 + \theta_5\theta_4\theta_3$, defining the first $\theta_4$-many steps as in~\ref{cohbas}, but use $\eta_\rho:=\theta_5 \rho$ for $\rho<\theta_4\theta_3$, a bijection $g\colon \theta_5\to \{0,1\}\times \theta_4\times \theta_3$, and a function $t\colon \theta_4\theta_3\to \theta_4$ such that $\{\rho<\theta_4\theta_3\colon t(\rho)=\alpha\}$ is cofinal in $\theta_4\theta_3$ for all $\alpha<\theta_4$.

For $\rho<\theta_4\theta_3$ and $\alpha<\theta_4$, choose
\begin{enumerate}[label=(F\arabic*)]
\setcounter{enumi}{-1}
   \item\label{F0} a collection $\set{\Qnm_{\alpha,\zeta}^\rho}{\zeta<\theta_5}$ of nice $\Por_{\alpha,\theta_\rho}$-names for posets of the form $\Loc^N$ for some transitive model $N$ of ZFC with $|N|<\theta_{1}$
   such that, for any $\Por_{\alpha,\theta_\rho}$-name $\dot F$ of a~subset of $\baire$ of size ${<}\theta_{1}$, there is some $\zeta<\theta_5$ such that, in $V_{\alpha,\theta_\rho}$, $\Qnm^\rho_{\alpha,\zeta} = \Loc^N$ for some $N$ containing $\dot F$;
   \item\label{F1} a collection $\set{\Qnm_{\alpha,\zeta}^\rho}{\zeta<\theta_5}$ of nice $\Por_{\alpha,\theta_\rho}$-names for posets of the form $\Dor^N$ for some transitive model $N$ of ZFC with $|N|<\theta_{2}$
   such that, for any $\Por_{\alpha,\theta_\rho}$-name $\dot F$ of a~subset of $\baire$ of size ${<}\theta_{2}$, there is some $\zeta<\theta_5$ such that, in $V_{\alpha,\theta_\rho}$, $\Qnm^\rho_{\alpha,\zeta} = \Dor^N$ for some $N$ containing $\dot F$; and
\end{enumerate}
\begin{enumerate}[label=(F${}^\rho$)]
   \item\label{Frho} enumeration $\set{\dot{f}_{\zeta}^\rho}{\zeta<\theta_5}$ of all the nice $\Por_{t(\rho),\theta_\rho}$-names for all the members of $(\omega\menos\{0\})^\omega$,
\end{enumerate}
and set:
\begin{enumerate}[label=\rm (C\arabic*)]
\setcounter{enumi}{1}

    \item if $\xi=\theta_\rho+4\varepsilon$ for some $\varepsilon<\theta_5$, put $\Delta(\xi):=t(\rho)$ and
    $\Qnm^*_{\xi}=\Qor_{\dot f^\rho_{\varepsilon}}^{V_{\Delta(\xi),\xi}}$;

    \item if $\xi=\theta_\rho+4\varepsilon+1$ for some $\varepsilon<\theta_5$, put $\Delta(\xi):=t(\rho)$ and
    $\Qnm^*_{\xi}=\Bor^{V_{\Delta(\xi),\xi}}$; and

    \item if $\xi=\theta_\rho+4\varepsilon+2$ for some $\varepsilon<\theta_5$, put $\Delta(\xi):=g_1(\varepsilon)$ and $\Qnm^*_{\xi}=\Qnm^{\rho}_{g(\varepsilon)}$.
\end{enumerate}
The construction is indeed a ${<}\theta_2$-uf-extendable iteration. We can now show that the construction does what we want, apart from
keeping $\add(\Nwf)$ small.

Notice that $\Por$ forces~\ref{SN1}--\ref{SN4} and~\ref{SN7}. This can proved exactly as in~\autoref{appl:I}. The proof of the rest of the claims is the same as presented in~\cite[Thm.~4.9]{CMR2} however, to ensure the paper is self-contained we include their proofs.

Directly by applying~\autoref{b12} to $\seq{\Por_{\theta_4,\xi}}{ \xi\leq \pi}$ and $\seq{\Por_{\alpha,\pi}}{ \alpha\leq \theta_4}$, respectively, we obtain that $\Por$ forces $\theta_3\leqT\Cbf^\perp_\SNwf$ and $\theta_4\leqT \Cbf^\perp_\SNwf$, so \ref{SN5}~is done.

To conclude it remains to see~\ref{SN6}.  Work in $V_{\gamma,\pi}$. Let $D\subseteq\baire$ be the set of all increasing functions. For each $f\in D$ let $f'\in\baire$ be defined by $f'(i) := f((i+1)^2)$. Since $\theta_4\theta_3\eqT \theta_3$, we construct a Tukey connection $\Phi_-\colon \SNwf\to (\theta_4\times\theta_4\theta_3)^D$, $\Phi_+\colon (\theta_4\times\theta_4\theta_3)^D \to \SNwf$.

For $A\in\SNwf$, we can find $\seq{\tau^A_f}{f\in D}\subseteq (2^{<\omega})^\omega$ such that $\hgt_{\tau^A_f} = f'$ and $A\subseteq \bigcap_{f\in D}[\tau^A_f]_\infty$. By~\autoref{realint}, for each $f\in D$ find $(\alpha^A_f,\rho^A_f)\in \theta_4\times\theta_4\theta_3$ such that $f,\tau^A_f\in\smash{V_{\alpha^A_f,\theta_{\rho^A_f}}}$. So set $\Phi_-(A):= \seq{(\alpha^A_f,\rho^A_f)}{f\in D}$.

Whenever $\xi = \theta_\rho + 4\varepsilon +1$ for some $\rho<\theta_4\theta_3$ and $\varepsilon<\theta_5$, let $\sigma^*_{\xi} \in 2^{f^\rho_\varepsilon}$ be the $\Qor_{f^\rho_\varepsilon}$-generic real over $V_{\Delta(\xi),\xi}$ added in $V_{\Delta(\xi),\xi+1}$.
Let $z=\seq{(\beta_f,\varrho_f)}{f\in D}$ in $(\theta_4\times \theta_4\theta_3)^D$. For each $f\in D$, find $\varrho'_f\geq \varrho_f$ in $\theta_4\theta_3$ such that $t(\varrho'_f) = \beta_f$. When $f\in \smash{V_{\beta_f,\theta_{\varrho'_f}}}$, find $\varepsilon_f<\theta_5$ such that $f=f^{\varrho'_f}_{\varepsilon_f}$, and let $\sigma_f := \sigma^*_{\xi_f}$ where $\xi_f:= \theta_{\varrho'_f}+3\varepsilon_f+1$, otherwise let $\sigma_f$ be anything in~$2^f$. Set $\Phi_+(z) := \bigcap_{f\in D}\bigcup_{n<\omega}[\sigma_f(n)]$, which is clearly in $\SNwf$.

It remains to show, by using the notation above, that $\Phi_-(A)\leq z$ implies $A\subseteq \Phi_+(z)$. If $\Phi_-(A)\leq z$, i.e.\ $\alpha^A_f\leq \beta_f$ and $\rho^A_f\leq \varrho_f$ for all $f\in D$, then $f,\tau^A_f\in V_{\beta_f,\varrho'_f}$, so $\sigma_f = \sigma^*_{\xi_f}$ and $[\tau^A_f]_\infty \subseteq \bigcup_{n<\omega}[\sigma_f(n)]$. Therefore, $A\subseteq \Phi_+(z)$.
\end{proof}

\section{Open problems}\label{sec:s6}

As we mentioned in~\autoref{sec:intro}, it is still unknown.

\begin{question}[{\cite[Sec.~7]{CMR2}}]\label{Qlowcf}
In~$\thzfc$,  which of the cardinals\/ $\bfrak$, $\dfrak$, $\supcof$, $\non(\Nwf)$ and\/ $\cof(\Nwf)$ are lower bounds of\/ $\cof(\SNwf)$?
\end{question}

In~\cite[Cor.~4.22]{CarMej23}, it was proved that if $\dfrak\leq\cof(\SNwf)$ then $\cov(\Mwf)<\cof(\SNwf)$, so whether $\dfrak\leq\cof(\SNwf)$ is true in $\thzfc$, we will have $\cov(\Mwf)<\cof(\SNwf)$. This would answer the following open problem:

\begin{question}[Yorioka~{\cite{Yorioka}}]
Does\/ $\thzfc$ prove\/ $\aleph_1<\cof(\SNwf)$?
\end{question}

Concerning Yorioka ideals, it remains unknown.

\begin{question}[{\cite[Sec.~7]{CM}}]
Does\/ $\thzfc$ prove\/ $\supcof=\cof(\Nwf)$ and\/ $\minadd = \add(\Nwf)$?
\end{question}

\subsection*{Acknowledgments}

The authors express their gratitude to Dr.\ Diego Mej\'ia for pointing us out that the $\sigma$-$\bar\rho$-linked property does not keep $\cov(\Nwf)$ small in generic extensions because random forcing is $\sigma$-$\bar\rho$-linked.

{\small
\bibliography{bibli}

\begin{thebibliography}{CMRM24}

\bibitem[BCM21]{BCM}
J\"{o}rg Brendle, Miguel~A. Cardona, and Diego~A. Mej\'{\i}a.
\newblock Filter-linkedness and its effect on preservation of cardinal characteristics.
\newblock {\em Ann. Pure Appl. Logic}, 172(1):Paper No. 102856, 30, 2021.

\bibitem[BCM23]{BCM2}
J\"{o}rg Brendle, Miguel~A. Cardona, and Diego~A. Mej\'ia.
\newblock Separating cardinal characteristics of the strong measure zero ideal.
\newblock Preprint, \href{https://arxiv.org/abs/2309.01931}{arXiv:2309.01931}, 2023.

\bibitem[BF11]{BrF}
J{\"o}rg Brendle and Vera Fischer.
\newblock Mad families, splitting families and large continuum.
\newblock {\em J. Symbolic Logic}, 76(1):198--208, 2011.

\bibitem[BJ95]{BJ}
Tomek Bartoszy\'{n}ski and Haim Judah.
\newblock {\em Set theory. On the structure of the real line}.
\newblock A K Peters, Ltd., Wellesley, MA, 1995.

\bibitem[BJS93]{BWS}
Tomek Bartoszy\'{n}ski, Winfried Just, and Marion Scheepers.
\newblock Covering {G}ames and the {B}anach-{M}azur {G}ame: {$K$}-tactics.
\newblock {\em Canad. J. Math.}, 45(5):897--929, 1993.

\bibitem[BM14]{BrM}
J{\"o}rg Brendle and Diego~Alejandro Mej{\'{\i}}a.
\newblock Rothberger gaps in fragmented ideals.
\newblock {\em Fund. Math.}, 227(1):35--68, 2014.

\bibitem[Bre91]{Br}
J{\"o}rg Brendle.
\newblock Larger cardinals in {C}icho\'n's diagram.
\newblock {\em J. Symbolic Logic}, 56(3):795--810, 1991.

\bibitem[BS89]{B1S}
Andreas Blass and Saharon Shelah.
\newblock Ultrafilters with small generating sets.
\newblock {\em Israel J. Math.}, 65(3):259--271, 1989.

\bibitem[Car22]{cardona}
Miguel~A. Cardona.
\newblock On cardinal characteristics associated with the strong measure zero ideal.
\newblock {\em Fund. Math.}, 257(3):289--304, 2022.

\bibitem[Car23]{Car23}
Miguel~A. Cardona.
\newblock A friendly iteration forcing that the four cardinal characteristics of {$\mathcal{E}$} can be pairwise different.
\newblock {\em Colloq. Math.}, 173(1):123--157, 2023.

\bibitem[CM19]{CM}
Miguel~A. Cardona and Diego~A. Mej\'{\i}a.
\newblock On cardinal characteristics of {Y}orioka ideals.
\newblock {\em MLQ}, 65(2):170--199, 2019.

\bibitem[CM22]{CM22}
Miguel~A. Cardona and Diego~A. Mej\'{\i}a.
\newblock Forcing constellations of {C}icho\'n's diagram by using the {T}ukey order.
\newblock {\em Ky\={o}to Daigaku S\=urikaiseki Kenky\=usho K\=oky\=uroku}, 2213:14--47, 2022.

\bibitem[CM23]{CMlocalc}
Miguel~A. Cardona and Diego~A. Mej\'{\i}a.
\newblock Localization and anti-localization cardinals.
\newblock {\em Ky\={o}to Daigaku S\=urikaiseki Kenky\=usho K\=oky\=uroku}, 2261:47--77, 2023.
\newblock \href{https://arxiv.org/abs/2305.03248}{arXiv:2305.03248}.

\bibitem[CM25]{CarMej23}
Miguel~A. Cardona and Diego~A. Mejía.
\newblock More about the cofinality and the covering of the ideal of strong measure zero sets.
\newblock {\em Annals of Pure and Applied Logic}, 176(4):103537, 2025.

\bibitem[CMRM22]{CMR}
Miguel~A. Cardona, Diego~A. Mej\'{\i}a, and Ismael~E. Rivera-Madrid.
\newblock The covering number of the strong measure zero ideal can be above almost everything else.
\newblock {\em Arch. Math. Logic}, 61(5-6):599--610, 2022.

\bibitem[CMRM24]{CMR2}
Miguel~A. Cardona, Diego~A. Mej\'{\i}a, and Ismael~E. Rivera-Madrid.
\newblock Uniformity numbers of the null-additive and meager-additive ideals.
\newblock Preprint, \href{https://arxiv.org/abs/2401.15364}{arXiv:2401.15364}, 2024.

\bibitem[CMUZ24]{CMU}
Miguel~A. Cardona, Diego~A. Mejía, and Andrés~F. Uribe-Zapata.
\newblock A general theory of iterated forcing using finitely additive measures.
\newblock Preprint, \href{https://arxiv.org/abs/2406.09978}{arXiv:2406.09978}, 2024.

\bibitem[GJS93]{GJS}
Martin Goldstern, Haim Judah, and Saharon Shelah.
\newblock Strong measure zero sets without {C}ohen reals.
\newblock {\em J. Symbolic Logic}, 58(4):1323--1341, 1993.

\bibitem[GKMS22]{GKMS}
Martin Goldstern, Jakob Kellner, Diego~A. Mej\'{\i}a, and Saharon Shelah.
\newblock Cicho\'{n}'s maximum without large cardinals.
\newblock {\em J. Eur. Math. Soc. (JEMS)}, 24(11):3951--3967, 2022.

\bibitem[Jec03]{Je2}
Thomas Jech.
\newblock {\em Set {T}heory, the {T}hird {M}illennium {E}dition, {R}evised and {E}xpanded}.
\newblock Springer Monographs in Mathematics. Springer-Verlag, Berlin, 2003.

\bibitem[JS89]{JS89}
Haim Judah and Saharon Shelah.
\newblock {${\rm MA}(\sigma$}-centered): {C}ohen reals, strong measure zero sets and strongly meager sets.
\newblock {\em Israel J. Math.}, 68(1):1--17, 1989.

\bibitem[JS90]{JS}
Haim Judah and Saharon Shelah.
\newblock The {K}unen-{M}iller chart ({L}ebesgue measure, the {B}aire property, {L}aver reals and preservation theorems for forcing).
\newblock {\em J. Symbolic Logic}, 55(3):909--927, 1990.

\bibitem[Kam89]{Ka}
Anastasis Kamburelis.
\newblock Iterations of {B}oolean algebras with measure.
\newblock {\em Arch. Math. Logic}, 29(1):21--28, 1989.

\bibitem[Kec95]{Ke2}
Alexander~S. Kechris.
\newblock {\em Classical descriptive set theory}, volume 156 of {\em Graduate Texts in Mathematics}.
\newblock Springer-Verlag, New York, 1995.

\bibitem[Kel59]{Kelley59}
J.~L. Kelley.
\newblock Measures on {B}oolean algebras.
\newblock {\em Pacific J. Math.}, 9:1165--1177, 1959.

\bibitem[KM22]{KM21}
Lukas~Daniel Klausner and Diego~Alejandro Mej{\'\i}a.
\newblock Many different uniformity numbers of yorioka ideals.
\newblock {\em Archive for Mathematical Logic}, 61(5):653--683, 2022.

\bibitem[KO08]{KO08}
Shizuo Kamo and Noboru Osuga.
\newblock The cardinal coefficients of the ideal {$\mathcal{I}_f$}.
\newblock {\em Arch. Math. Logic}, 47(7-8):653--671, 2008.

\bibitem[KO14]{KO}
Shizuo Kamo and Noboru Osuga.
\newblock Many different covering numbers of {Y}orioka's ideals.
\newblock {\em Arch. Math. Logic}, 53(1-2):43--56, 2014.

\bibitem[Kun80]{kunen80}
Kenneth Kunen.
\newblock {\em Set theory}, volume 102 of {\em Studies in Logic and the Foundations of Mathematics}.
\newblock North-Holland Publishing Co., Amsterdam-New York, 1980.
\newblock An introduction to independence proofs.

\bibitem[Kun11]{kunen}
Kenneth Kunen.
\newblock {\em Set {T}heory}, volume~34 of {\em Studies in Logic (London)}.
\newblock College Publications, London, 2011.

\bibitem[Lav76]{LaverBorel}
Richard Laver.
\newblock On the consistency of {B}orel's conjecture.
\newblock {\em Acta Math.}, 137(3-4):151--169, 1976.

\bibitem[Mej13]{mejiamatrix}
Diego~Alejandro Mej{\'{\i}}a.
\newblock Matrix iterations and {C}ichon's diagram.
\newblock {\em Arch. Math. Logic}, 52(3-4):261--278, 2013.

\bibitem[Mej19]{mejiavert}
Diego~A. Mej\'{\i}a.
\newblock Matrix iterations with vertical support restrictions.
\newblock In {\em Proceedings of the 14th and 15th {A}sian {L}ogic {C}onferences}, pages 213--248. World Sci. Publ., Hackensack, NJ, 2019.

\bibitem[Mil82]{Mi1982}
Arnold~W. Miller.
\newblock A characterization of the least cardinal for which the {B}aire category theorem fails.
\newblock {\em Proc. Amer. Math. Soc.}, 86(3):498--502, 1982.

\bibitem[Osu08]{O08}
Noboru Osuga.
\newblock {The Cardinal Invariants of Certain Ideals Related to the Strong Measure Zero Ideal}.
\newblock {\em Ky\={o}to Daigaku S\=urikaiseki Kenky\=usho K\=oky\=uroku}, 1619:83--90, 2008.
\newblock \href{https://www.kurims.kyoto-u.ac.jp/~kyodo/kokyuroku/contents/1619.html}{https://www.kurims.kyoto-u.ac.jp/~kyodo/kokyuroku/contents/1619.html}.

\bibitem[Paw90]{P90}
Janusz Pawlikowski.
\newblock Finite support iteration and strong measure zero sets.
\newblock {\em J. Symbolic Logic}, 55(2):674--677, 1990.

\bibitem[Tal80]{Tal98}
Michel Talagrand.
\newblock Compacts de fonctions mesurables et filtres non mesurables.
\newblock {\em Studia Math.}, 67(1):13--43, 1980.

\bibitem[UZ24]{U23}
Andrés~F. Uribe-Zapata.
\newblock The intersection number for forcing notions.
\newblock {\em Ky\={o}to Daigaku S\=urikaiseki Kenky\=usho K\=oky\=uroku}, 2290:1--17, 2024.
\newblock \href{https://arxiv.org/abs/2401.14552}{arXiv:2401.14552}.

\bibitem[Yor02]{Yorioka}
Teruyuki Yorioka.
\newblock The cofinality of the strong measure zero ideal.
\newblock {\em J. Symbolic Logic}, 67(4):1373--1384, 2002.

\end{thebibliography}
\bibliographystyle{alpha}
}


\end{document}